\pdfoutput=1
\documentclass[11pt]{amsart}
\usepackage{latexsym,amssymb,graphicx,amscd,amsmath}
\usepackage{graphicx}
\usepackage{epstopdf}
\usepackage{tipa}
\usepackage{tipx}

\newtheorem{thm}{Theorem}[section]
\newtheorem{lem}[thm]{Lemma}
\newtheorem{prop}[thm]{Proposition}
\newtheorem{cor}[thm]{Corollary}
\newtheorem{de}[thm]{Definition}
\newtheorem{rem}[thm]{Remark}

\newtheorem{note}[thm]{Notation}

\newcommand{\br}{{\mathbb{R}}}
\newcommand{\bz}{{\mathbb{Z}}}
\newcommand{\bn}{{\mathbb{N}}}
\newcommand{\cb}{{\mathcal{B}}}
\newcommand{\cd}{{\mathcal{D}}}
\newcommand{\ce}{{\mathcal{E}}}
\newcommand{\cs}{{\mathcal{S}}}

\newcommand{\ttx}{{\left(\frac{1}{2x}\right)}}

\begin{document}

\title{Bilinear Forms on Skein Modules and Steps in Dyck Paths}

\author{Xuanting Cai}
\address{Mathematics Department\\
Louisiana State University\\
Baton Rouge, Louisiana}
\email{xcai1@math.lsu.edu}
\author{Toufik Mansour}
\address{Department of Mathematics\\
University of Haifa\\
31905 Haifa, Israel}
\email{toufik@math.haifa.ac.il}

\date{\today}

\begin{abstract}
We use Jones-Wenzl idempotents to construct bases for the relative Kauffman bracket skein module of a square with $n$ points colored $1$
and one point colored $h$.
We consider a natural bilinear form on this skein module.
We calculate the determinant of the matrix for this form with respect to the natural basis.
We reduce the computation to count some steps in generalized Dyck paths.
Moreover, we relate our determinant to a determinant on semi-meanders.
\end{abstract}
\maketitle

\section{Introduction}
Quantum topology is the study of new invariants that arose after the discovery of the Jones polynomial in 1984 \cite{J}.
After using ideas from physics, Witten \cite{W} suggested how these invariants could be viewed more intrinsically and extended to 3-manifolds.
Vaughan Jones discovered his famous knot polynomial in \cite{J}, which triggered the developments relating knot theory, topological quantum field theory,
and statistical physics \cite{K1} \cite{W} \cite{TL}.
A central role in those developments was played by the Temperley-Lieb algebra.
In \cite{L1}, Lickorish constructed quantum invariants for 3-manifolds from the Temperley-Lieb algebra $TL_n$.
He used a natural bilinear form on $TL_n$. Our aim is to generalize this skein module and the bilinear form.
As a module, the Temperley-Lieb algebra $TL_n$ can be considered as a skein module of a square with $2n$ points on the boundary.
Then skein modules of a square with points on the boundary colored by non-negative integers are a natural generalization of $TL_n$.
The same methods \cite{C} that the first author has employed in studying $TL_n$ may be adapted to this more general situation.
In order to understand this, we consider the skein module of a square $S(I\times I,n,h)$ with $n$ points on $I\times\{0\}$ and one point colored $h$ on $I\times\{1\}$. 
We consider the natural generalization of Lickorish's bilinear form on $S(I\times I,n,h)$,
and define the determinant of the bilinear form with respect to a natural basis $\cb_n^h$.
We find that the determinant that we calculated is related to a semi-meander determinant that was suggested to the first author by his advisor Patrick Gilmer.
This semi-meander determinant is different from the semi-meander determinant defined by Di Francesco in \cite{F}.

We compute the determinant of the bilinear form using an orthogonal basis $\cd_n^h$, which is motivated by \cite{BHMV}.
The transform matrix between $\cb_n^h$ and $\cd_n^h$ is an upper triangular matrix with $1$'s on the diagonal.
So the determinant we get by using the basis $\cd_n^h$ is the same as the determinant we get by using the basis $\cb_n^h$.
In the calculation, we set up a correspondence between the elements of $\cd_n^h$ and generalized Dyck paths on $\br^2$.
The problem is then reduced to count certain steps in all generalized Dyck paths from $(0,0)$ to $(n,h)$.
In Section \ref{ACR}, we present two different proofs for the counting problem.
The first method is geometric.
It is a generalization of the method used by Di Francesco for the case $h=0$.
The second method uses generating functions.

\section{The Skein Module $S(I\times I,n,h)$}
\label{TSM}
Let $F$ be an oriented surface with a finite collection of points specified in its boundary $\partial F$.
A link diagram in the surface $F$ consists of finitely many arcs and closed curves in $F$,
with a finite number of transverse crossings, each assigned over or under information.
The endpoints of the arcs must be the specified points on $\partial F$.
We define the skein of $F$ as follows:

\begin{de}
Suppose $A$ is a variable.
Let $\Lambda$ be the ring $\bz[A,A^{-1}]$ localized by inverting the multiplicative set generated by the elements of $\{A^n-1\mid n\in\bz^+\}$.
If we specify $n$ points on $\partial F$, where $n$ could be $0$,
the linear skein $\cs(F,n)$ is  the module of formal linear sums over $\Lambda$ of $n$-endpoint tangle
diagrams in $F$ with their end points identified to the specified points on $\partial F$, quotiented by the submodule generated by the skein relations:
\begin{enumerate}
\item $L \cup U= \delta L$, where $U$ is a trivial knot, $L$ is a link in $F$ and $\delta=(-A^{-2}-A^2)$;
\item $\begin{minipage}{0.2in}\includegraphics[width=0.2in]{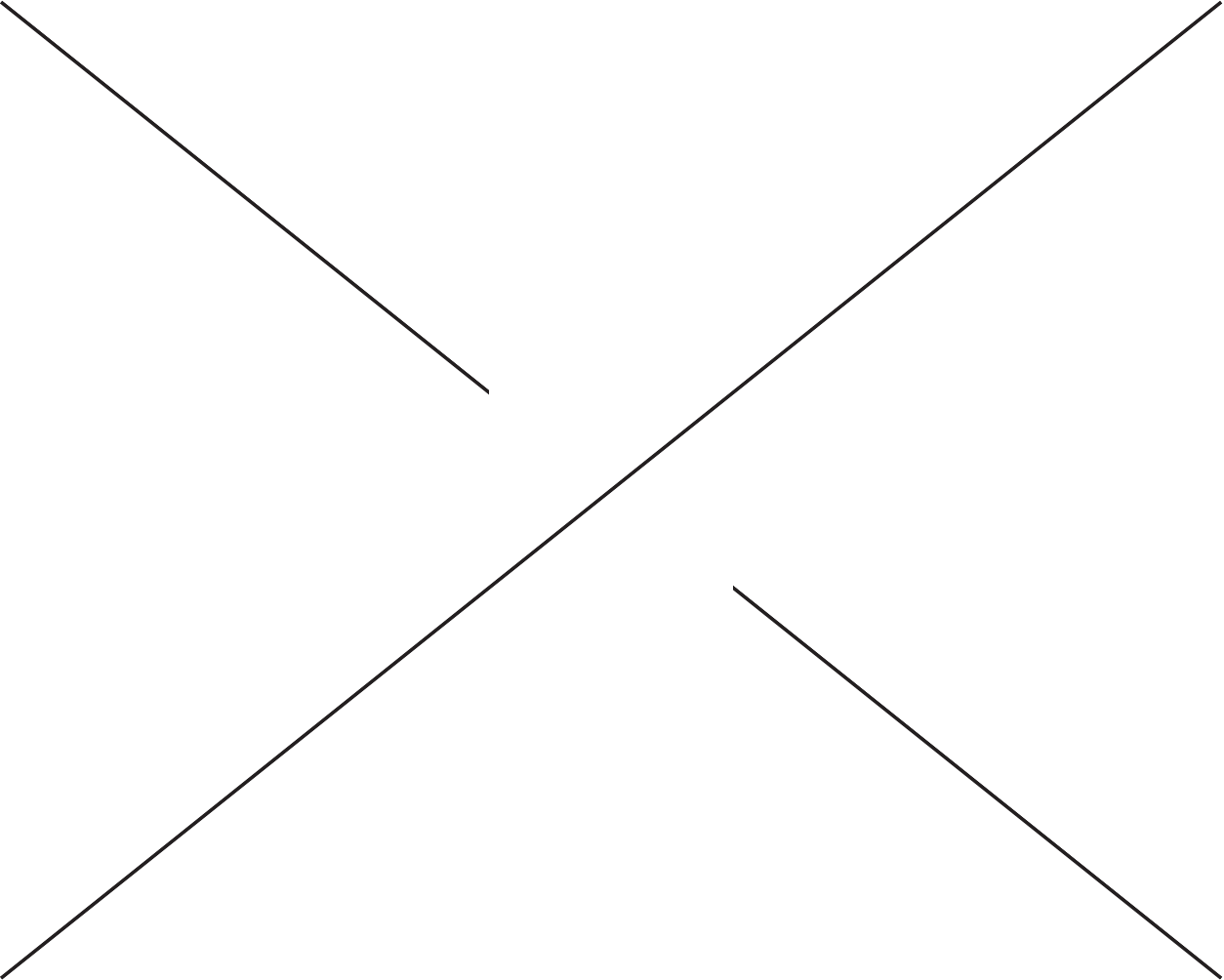}\end{minipage}=\ A^{-1}\  \begin{minipage}{0.2in}\includegraphics[width=0.2in]{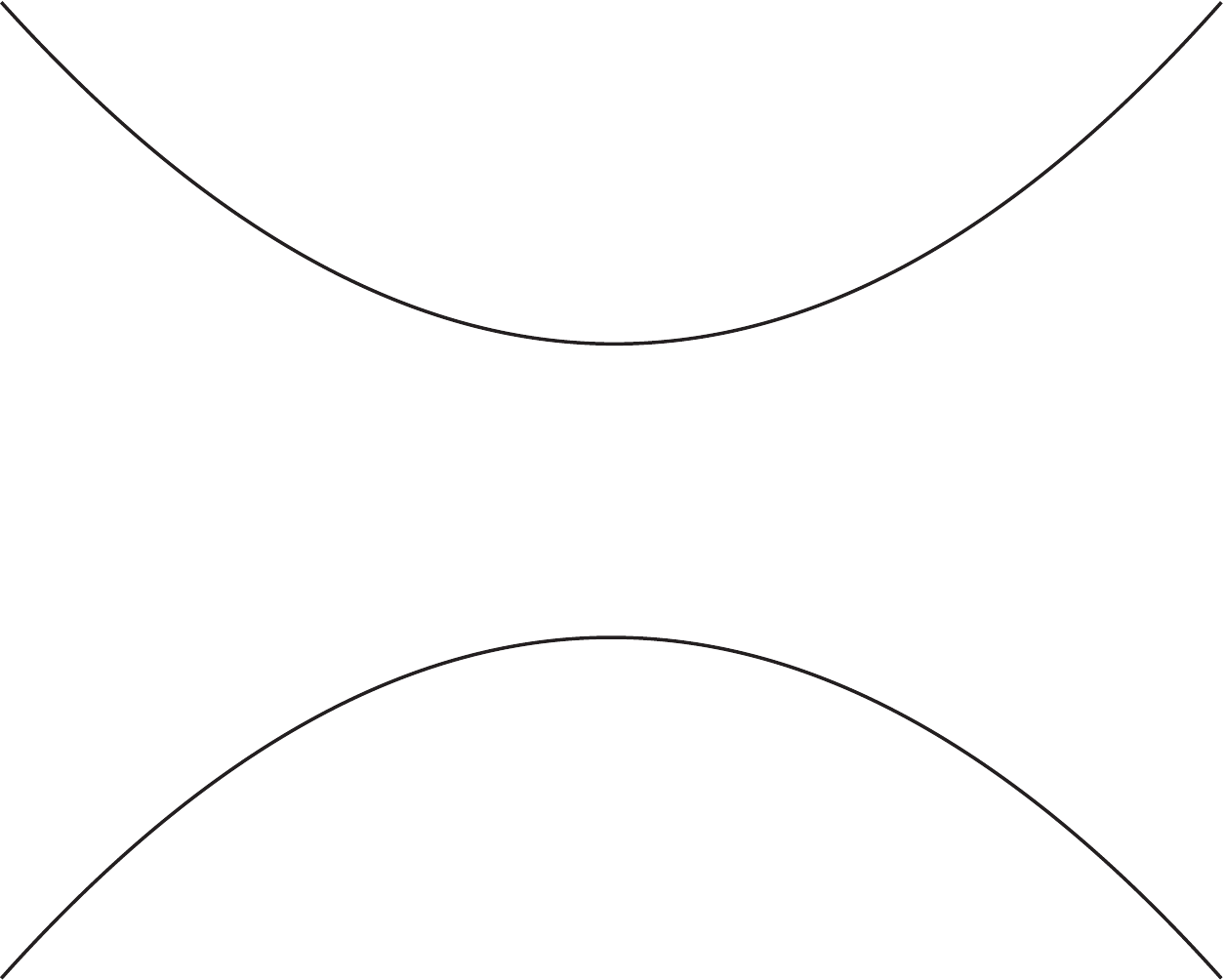}\end{minipage}\ +\ A\ 	 \begin{minipage}{0.2in}\includegraphics[width=0.2in]{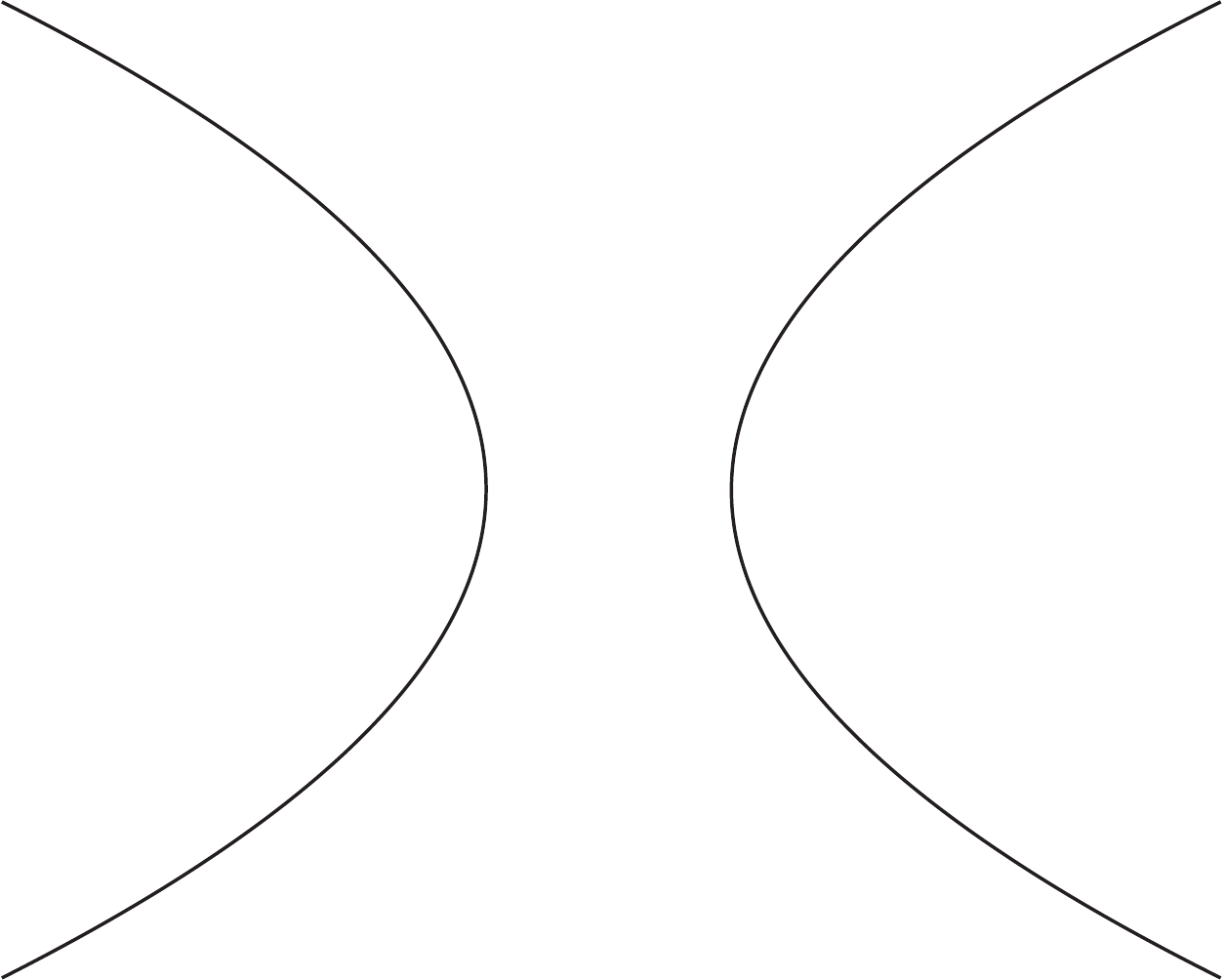}\end{minipage}$ .
\end{enumerate}
\end{de}

\begin{de}
If we take $F=I\times I$ with $n$ points on $I\times{0}$ and $n$ points on $I\times{1}$,
then the skein module obtained is called the $n$th Temperley-Lieb algebra $TL_n$.
There is a natural multiplication on $TL_n$,
it is done by putting the first element on top of the second,
see Figure \ref{multiplication}.
\begin{figure}[h]

\includegraphics[width=0.8in, height=1in]{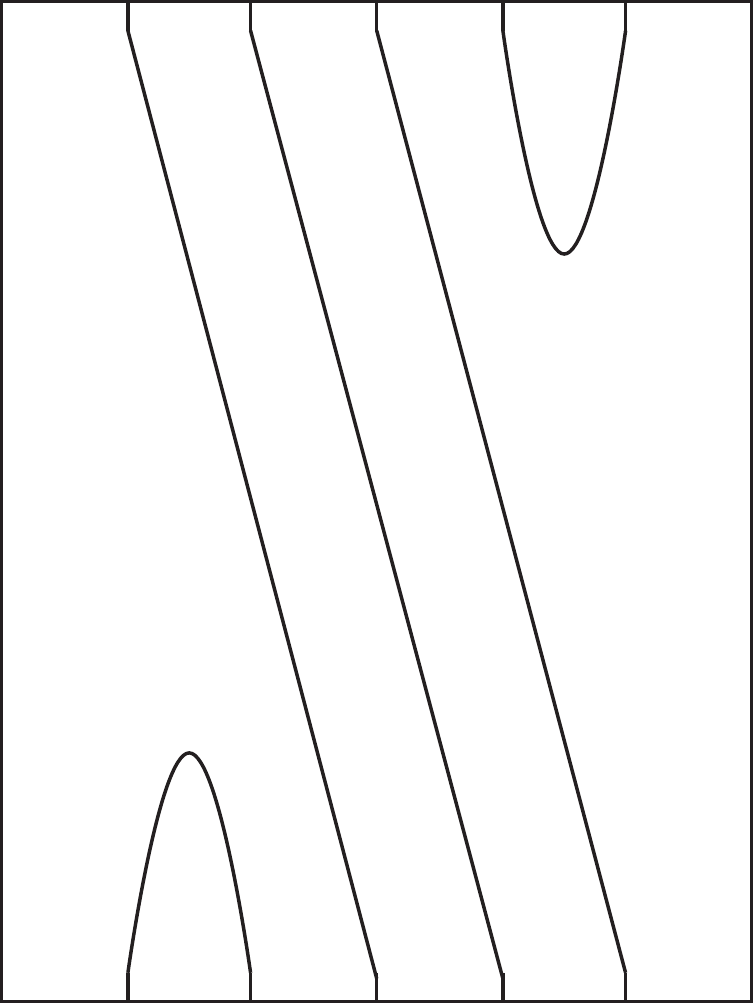}
\includegraphics[width=0.2in, height=1in]{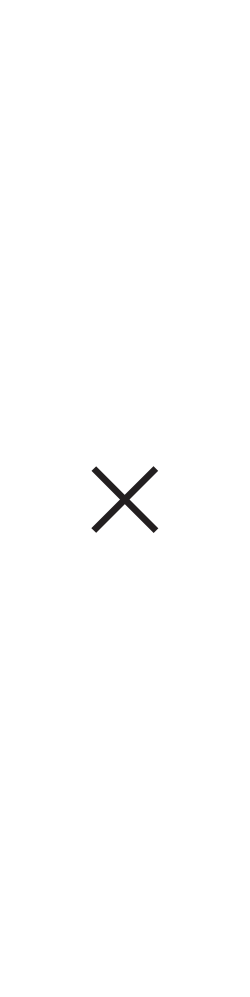}
\includegraphics[width=0.8in, height=1in]{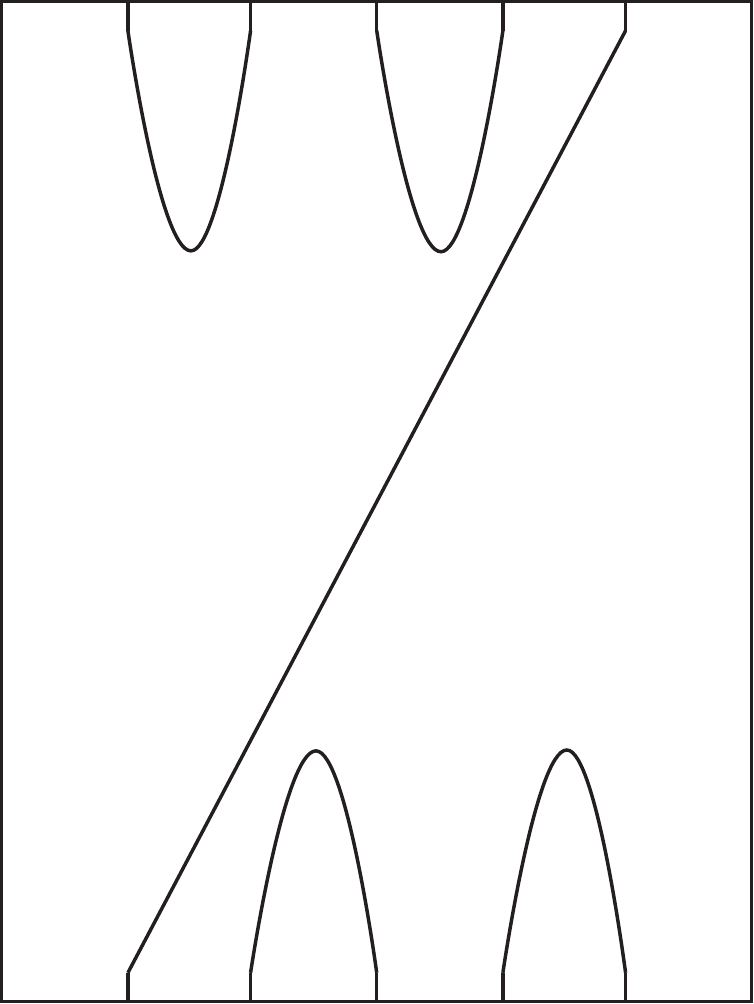}
\includegraphics[width=0.2in, height=1in]{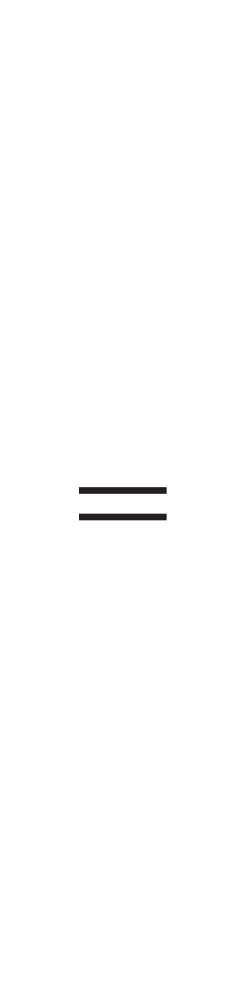}
\includegraphics[width=0.8in, height=1in]{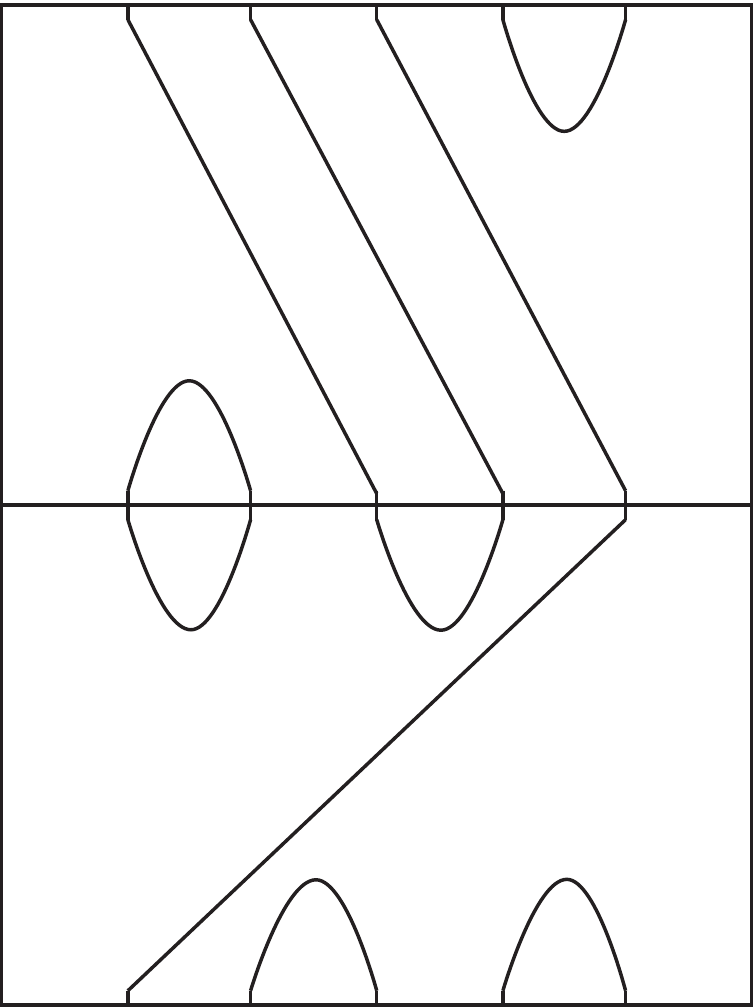}
\includegraphics[width=0.2in, height=1in]{vequal.pdf}
\includegraphics[width=0.2in, height=1in]{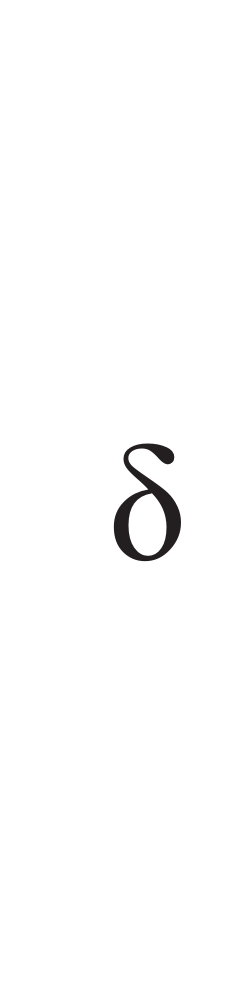}
\includegraphics[width=0.8in, height=1in]{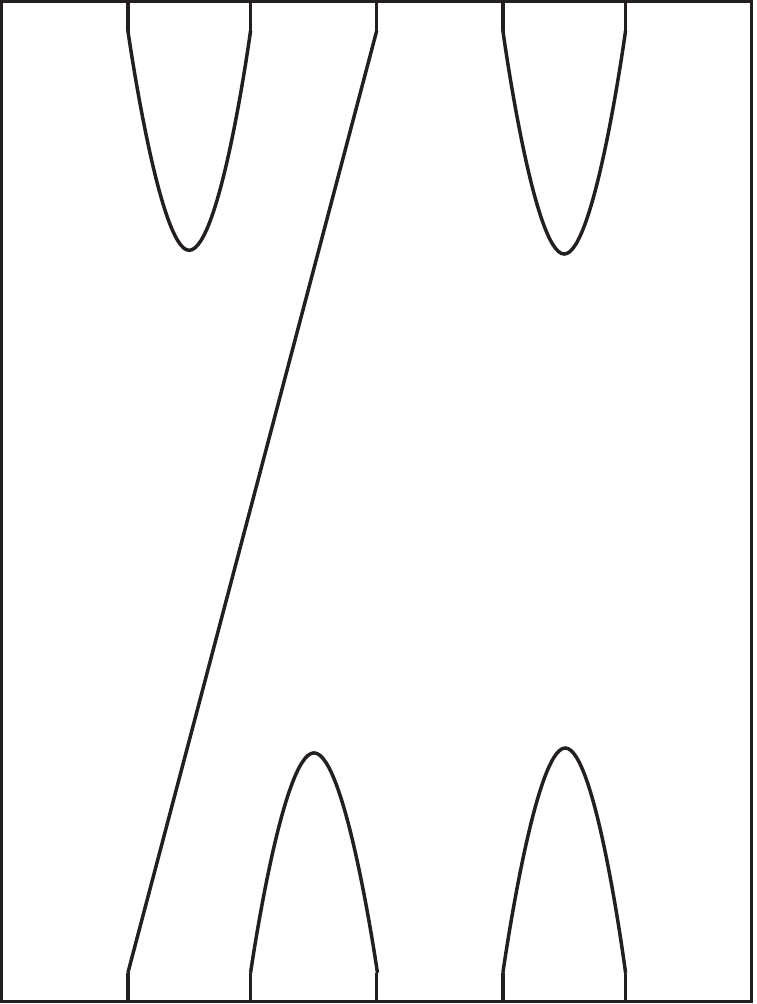}
\caption{Multiplication on Temperley-Lieb algebra}
\label{multiplication}
\end{figure}
\end{de}

There is a well known basis for $TL_n$,
which consists of crossing free figures.
We denote this basis by $\cb_n$.
Moreover, elements $\{1,e_i\}_{i=1,\ldots,n-1}$ are generators for $TL_n$ as an algebra,
where $1$ and $e_i$'s as presented in Figure \ref{generator}.
We put an integer $j$ beside a string to denote $j$ parallel strings.

\begin{figure}[h]

\[ 1 =
\begin{array}{c}
\includegraphics[width=0.8in, height=1in]{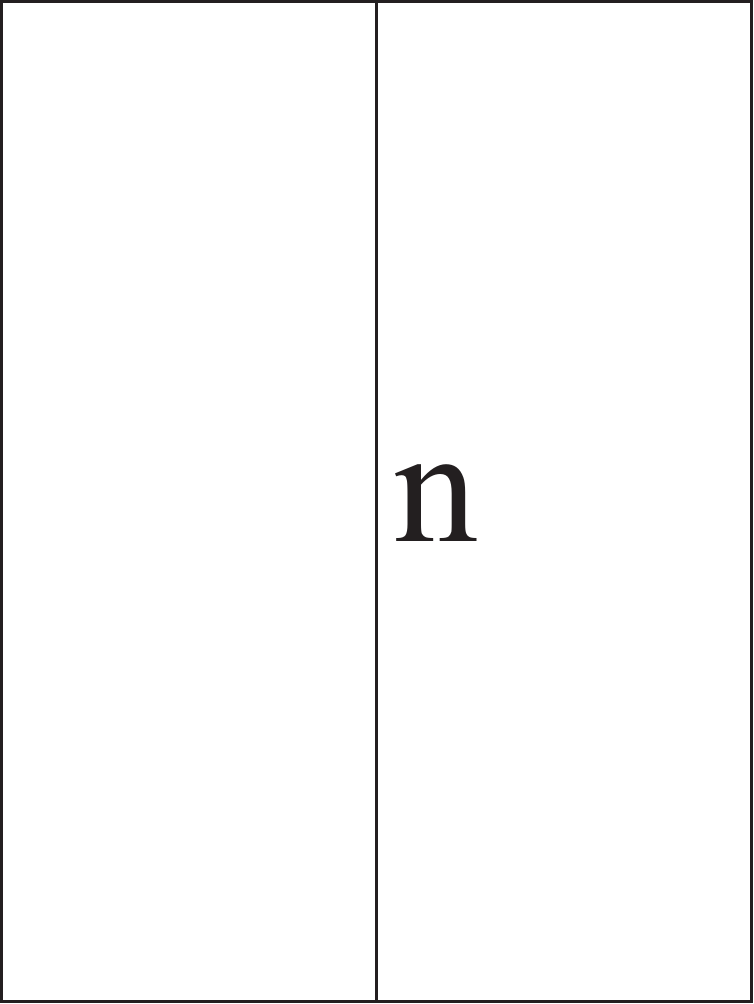}
\end{array}\hskip 0.5in
   e_i =
\begin{array}{c}
\includegraphics[width=0.8in, height=1in]{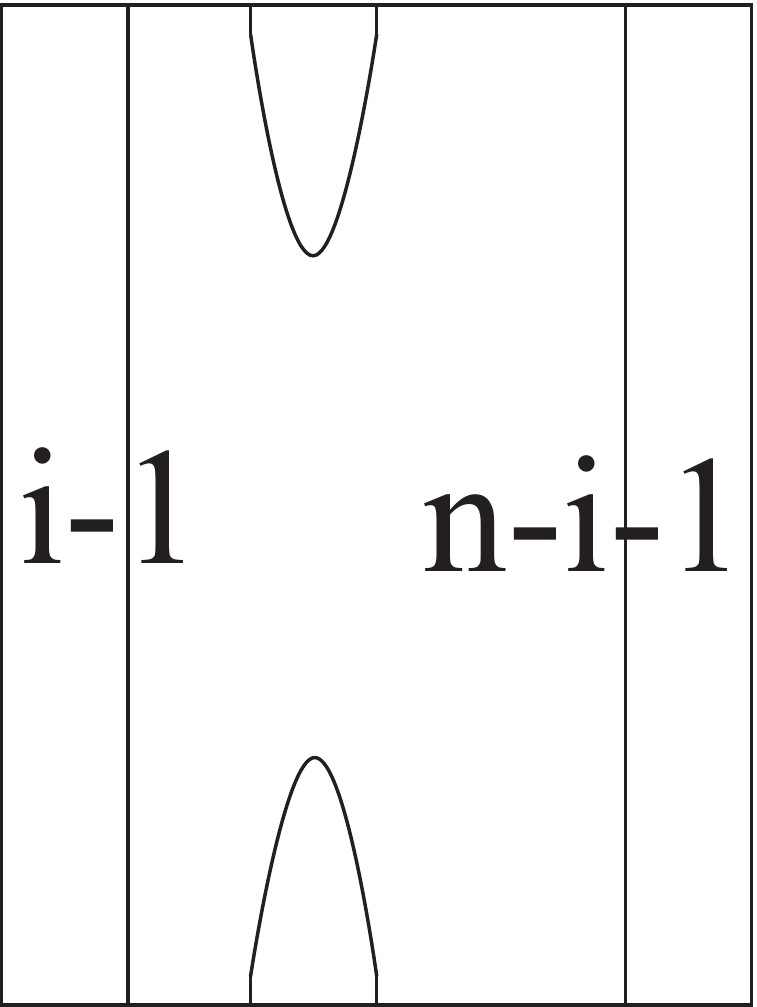}
\end{array}
\]
\caption{Generators for $TL_n$, where $i$ means the $i$th point}
\label{generator}
\end{figure}

A special element in $TL_n$, called the $n$th Jones-Wenzl idempotent, plays the key role in constructing 3-manifolds invariants, which are introduced by Wenzl in \cite{Wen}. Here, we follow the definition of Lickorish's in \cite{L2}, which characterizes Jones-Wenzl idempotents with the following property:

\begin{prop}
\label{JWI}
There is a unique element $f_n\in TL_n$, called the $n$th Jones-Wenzl
idempotent, such that
\begin{enumerate}
\item $f_ne_i=0=e_if_n$ for $1\leq i\leq n-1$;
\item $(f_n-1)$ belongs to the subalgebra generated by $e_1,\ldots,e_{n-1}$;
\item $f_nf_n=f_n$.
\end{enumerate}
\end{prop}

\begin{rem}
If we close $f_n$ up and evaluate the diagram by Kauffman bracket,
the resulting value is denoted by $\Delta_n$.
It is well known that
\begin{equation}
\Delta_n=(-1)^n\frac{A^{2(n+1)}-A^{-2(n+1)}}{A^2-A^{-2}}.
\notag
\end{equation}
\end{rem}

We can also describe those Jones-Wenzl idempotents by recursive formula:
\begin{figure}[h]
\[
f_1 =
\begin{array}{c}
\includegraphics[width=1in, height=0.8in]{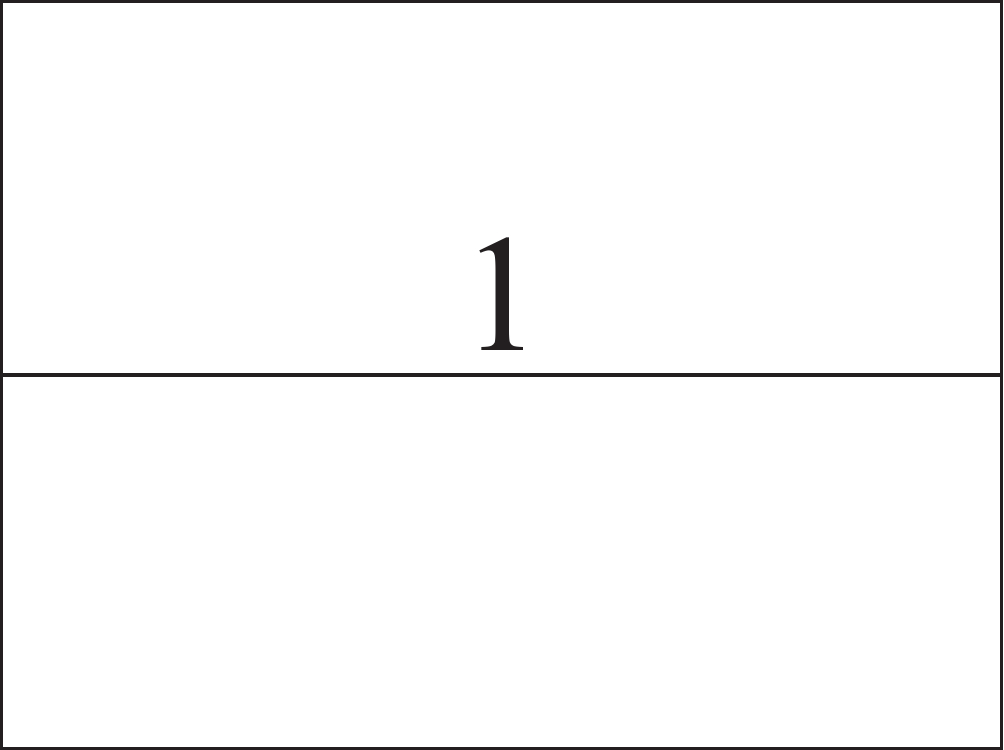}
\end{array}
\]
\end{figure}
\begin{figure}[h]
\includegraphics[width=1in, height=0.8in]{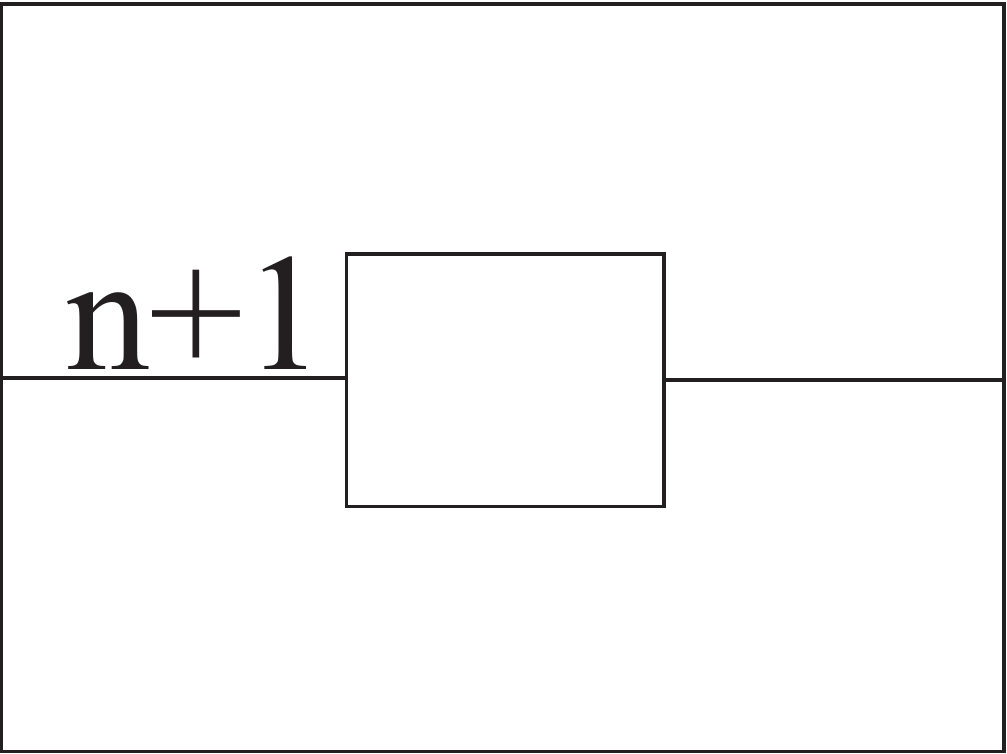}
\includegraphics[width=0.2in, height=0.8in]{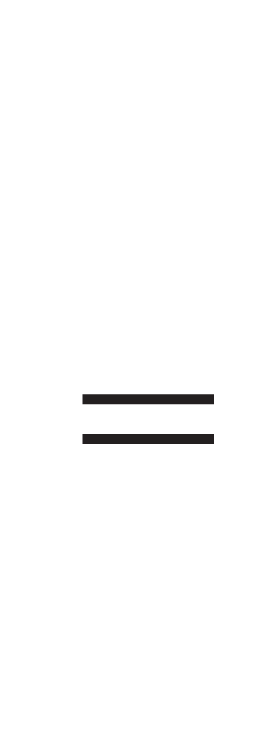}
\includegraphics[width=1in, height=0.8in]{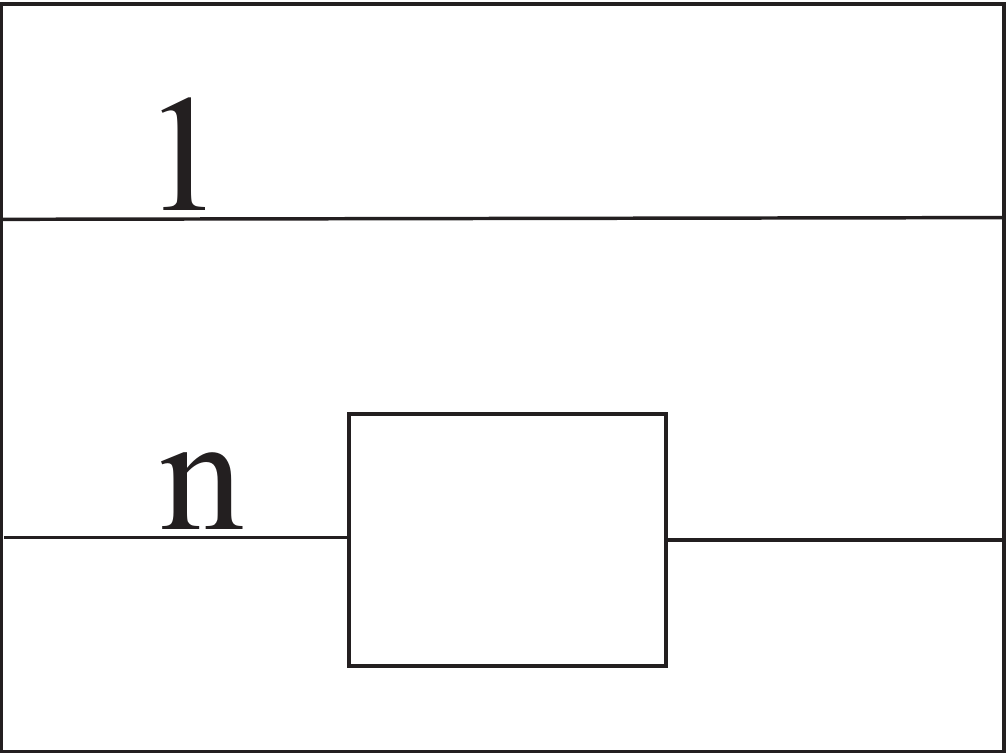}
\includegraphics[width=0.2in, height=0.8in]{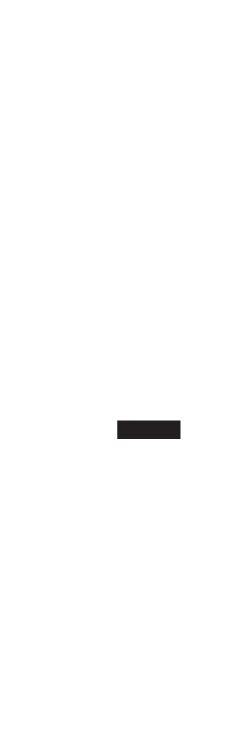}
\includegraphics[width=0.3in, height=0.8in]{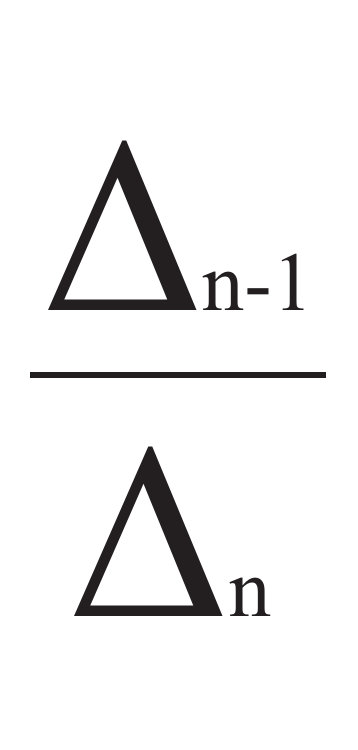}
\includegraphics[width=1in, height=0.8in]{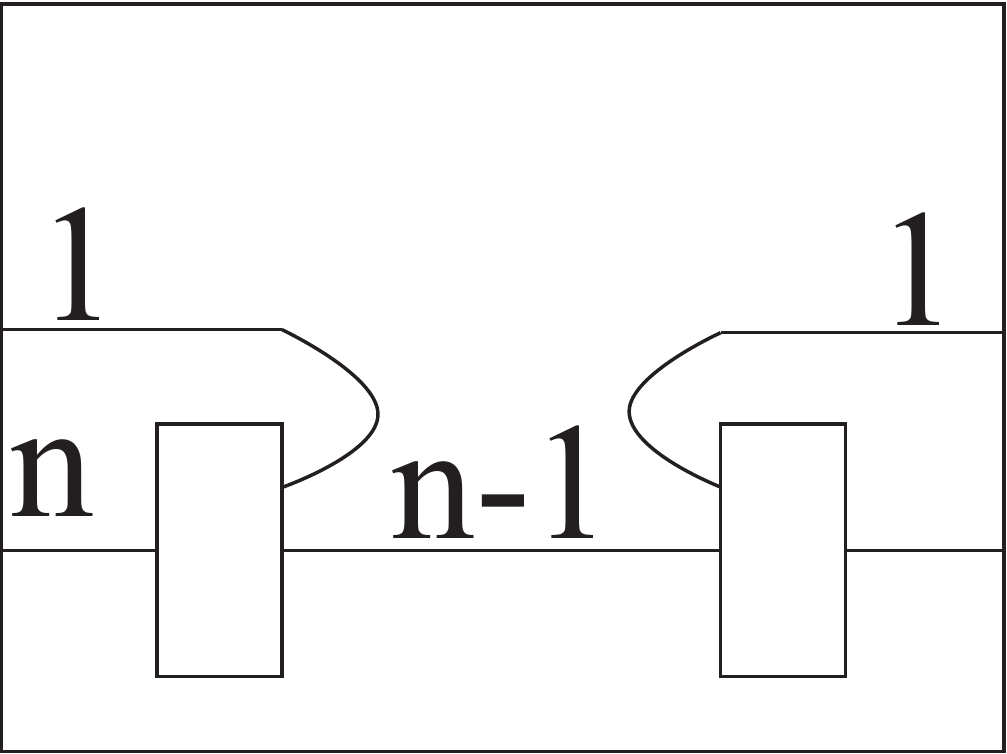}
\label{idempotent}
\caption{The box inside denotes the corresponding idempotent}
\label{recursiveformula}
\end{figure}

Now we introduce colored skein modules.

\begin{de}
Coloring a point with $n$ means assigning a non-negative integer $n$ to the point.
\end{de}

Now we define colored skein module by following Lickorish's idea \cite[Page 151]{L2}.
\begin{de}
Suppose we have $a_1+...+a_n$ points on $\partial F$.
We partition them into $n$ sets of $a_1,...,a_n$ consecutive points and put corresponding Jones-Wenzl idempotent just outside each diagram for each grouped points.
All this kind of elements form a subspace of $\cs(F,a_1+...+a_n)$.
We call this subspace the skein module $\cs(F;a_1,...,a_n)$ of $F$ with $n$ points on $\partial F$ colored by $a_1,...,a_n$.
\end{de}

Since $f_1$ is the identity of $TL_1$,
we can consider the normal skein module as the colored skein module of $\cs(F,n)$ with coloring $1, \ldots, 1$, i.e.
\begin{equation}
\cs(F,n)=\cs(F;1,...,1).
\notag
\end{equation}
Therefore, colored skein modules  generalize the definition of normal skein modules.

\begin{note}
We denote by $\cs(I\times I,n,h)$ the colored skein module in $\cs(I\times I,n+1)$ with $n$ points colored by $1$ and $1$ point colored by $h$.
We put the point colored by $h$ on $I\times\{1\}$ and the $n$ points colored by $1$ on $I\times \{0\}$.
Here $I=[0,1]$.
\end{note}

\begin{prop}
\label{generating}
There is a natural generating set $\cb_n^h$ for $\cs(I\times I,n,h)$ consisting of crossing free diagrams with no arc connecting two of the $h$ points in $I\times \{1\}$.
\end{prop}
\begin{proof}
For a crossing in a diagram, we just use the second skein relation to smooth it.
Then we can smooth all crossings in every diagram.
That means every diagram can be written as a linear sum of crossing free diagrams.
Therefore, we can get a generating set consisting of crossing free diagrams.
Moreover, by Proposition \ref{JWI}, a crossing free diagram with a segment connecting two of the set of $h$ points is $0$ in $\cs(I\times I,n,h)$.
Thus, the result follows.
\end{proof}

\begin{rem}
We will see later that actually $\cb_n^h$ is a basis of $\cs(I\times I,n,h)$.
So we will call $\cb_n^h$ the natural basis.
\end{rem}

We can also generalize the natural bilinear form on $TL_n$ to a bilinear form on $\cs(I\times I,n,h)$ as follows:
\begin{prop}
\label{bilinearform}
Suppose $A=\begin{minipage}{0.5in}\includegraphics[width=0.5in]{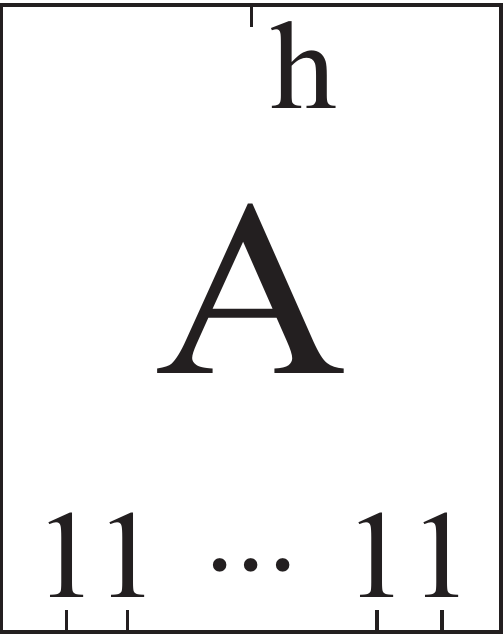}\end{minipage}$ and $B=\begin{minipage}{0.5in}\includegraphics[width=0.5in]{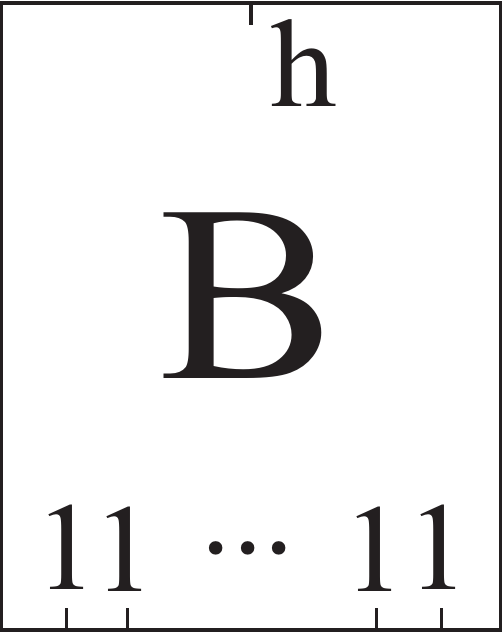}\end{minipage}$ are two elements in $\cs(I\times I,n,h)$.
We define a function:
\begin{equation}
G:\cs(I\times I,n,h)\times\cs(I\times I,n,h)\rightarrow\Lambda
\notag
\end{equation}
as
\begin{equation}
G(A,B)=\begin{minipage}{1.2in}\includegraphics[width=1.2in,height=0.8in]{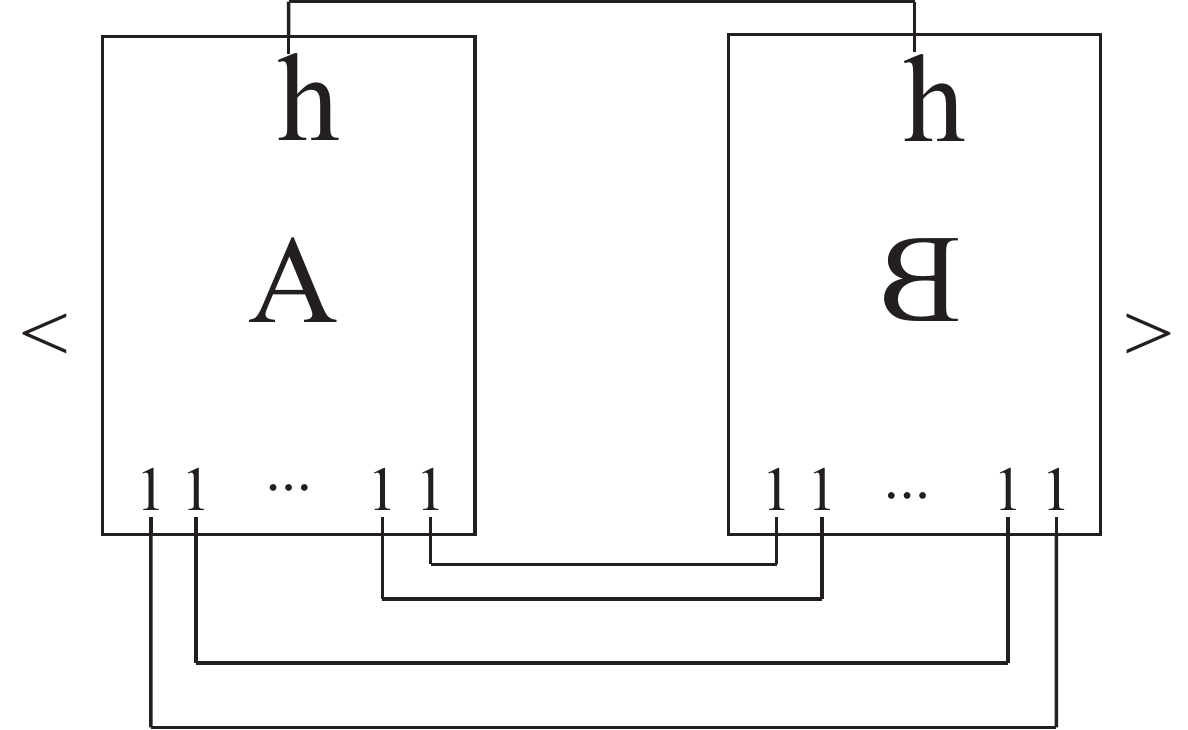}\end{minipage}\ .
\notag
\end{equation}
where $<>$ means we evaluate the diagram in the plane by Kauffman bracket.
This is a bilinear form on $\cs(I\times I,n,h)$.
\end{prop}
\begin{proof}
This is a standard skein theoretic argument.
\end{proof}

Suppose the matrix of the bilinear form $G$ with respect to the natural basis $\cb_n^h$ is $B$.
We will calculate the determinant of $B$.

\begin{thm}\label{main} We have
\begin{equation}
\det(B)=\Delta_h^{|\cd_n^h|}\prod_{k}(\frac{\Delta_k}{\Delta_{k-1}})^{\alpha^k_{(n,h)}},
\notag
\end{equation}
where $\alpha^k_{(n,h)}=\binom{n}{\frac{n+h+2k-2s}{2}}-\binom{n}{\frac{n+h+2k+2}{2}}$ for $s=\min\{k-1,h\}$.
\end{thm}

\section{A New Basis $\cd_n^h$}
\label{ANB}
We do not directly compute the determinant of $B$.
In fact, we construct a new basis $\cd_n^h$ of $\cs(I\times I,n,h)$,
which is orthogonal with respect to the bilinear form.
We then find the transform matrix between them.

Before we continue, let us set up some notations.

\begin{de}
Three non-negative integers $a, b, c$ are called admissible if they satisfy
\begin{enumerate}
\item a+b+c is even;
\item $|b-c|\leq a\leq b+c$.
\end{enumerate}
\end{de}

\begin{figure}[htp]
\includegraphics[width=0.8in, height=1in]{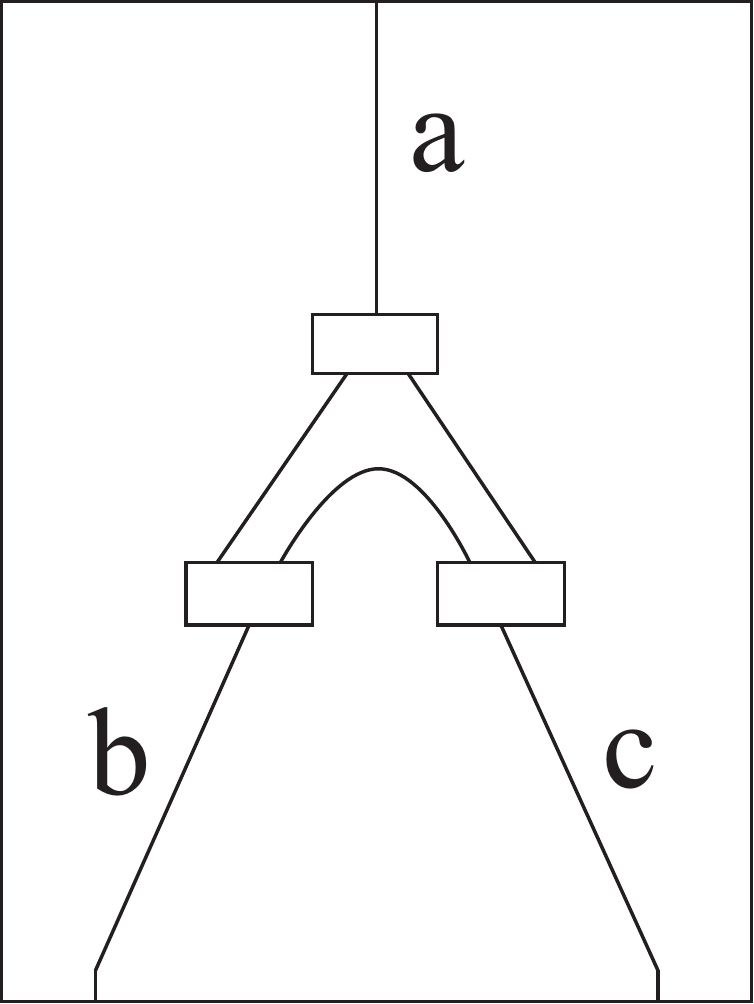}\hskip 0.5in
\includegraphics[width=0.8in, height=1in]{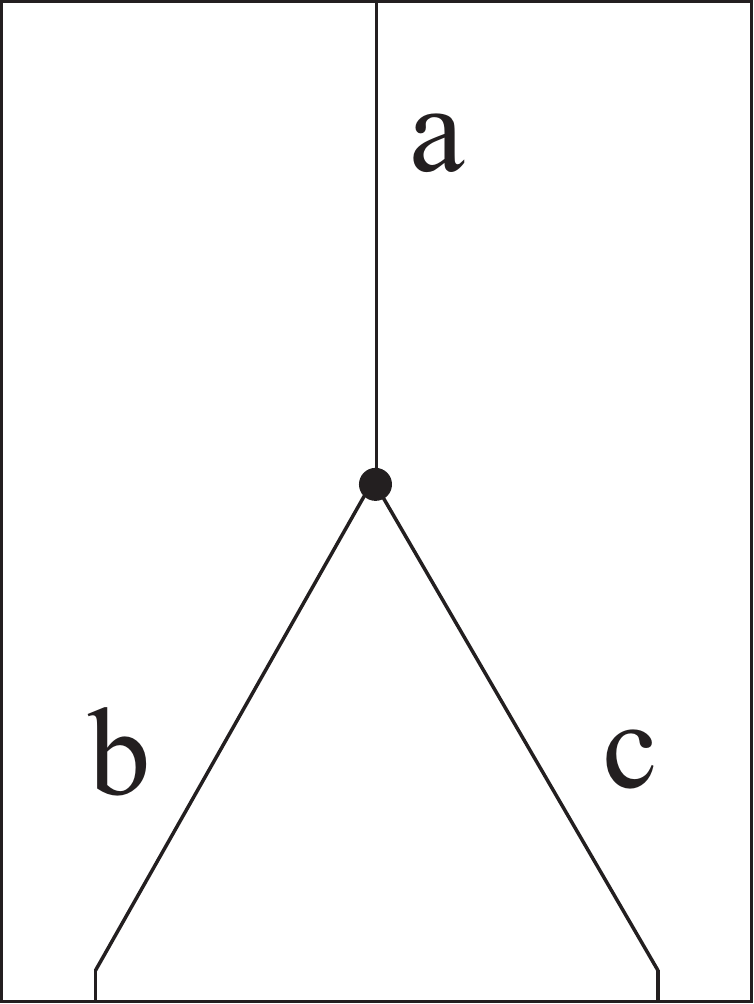}
\caption{We denote the figure on the left by the figure on the right}
\label{trivalent}
\end{figure}

\begin{prop}
The colored skein module $\cs(D^2;a,b,c)$ is 1 dimensional if $a, b, c$ are admissible,
$0$ dimensional otherwise.
\end{prop}
\begin{proof}
This is a standard result in skein theory, see, for example, \cite{L2} or \cite{K2}.
\end{proof}

\begin{rem}\label{rtrivalent}
The generator for $\cs(D^2;a,b,c)$ is the diagram on the left of Figure \ref{trivalent},
and we use a trivalent graph with a black dot as the diagram on the right to denote the generator.
\end{rem}

\begin{de}
\label{newbasis}
We define an element $D_{a_1, \ldots, a_n}$  of $\cs(I\times I,n,h)$ as in Figure \ref{dbasis},
where the black dot is as in Remark \ref{rtrivalent},
and $a_1, \ldots, a_n$ are non-negative integers satisfying the admissible condition at each black dot.
Let $\cd_n^h$ be the set of such elements.
\end{de}

\begin{figure}[htp]
\includegraphics[width=1.2in, height=1.5in]{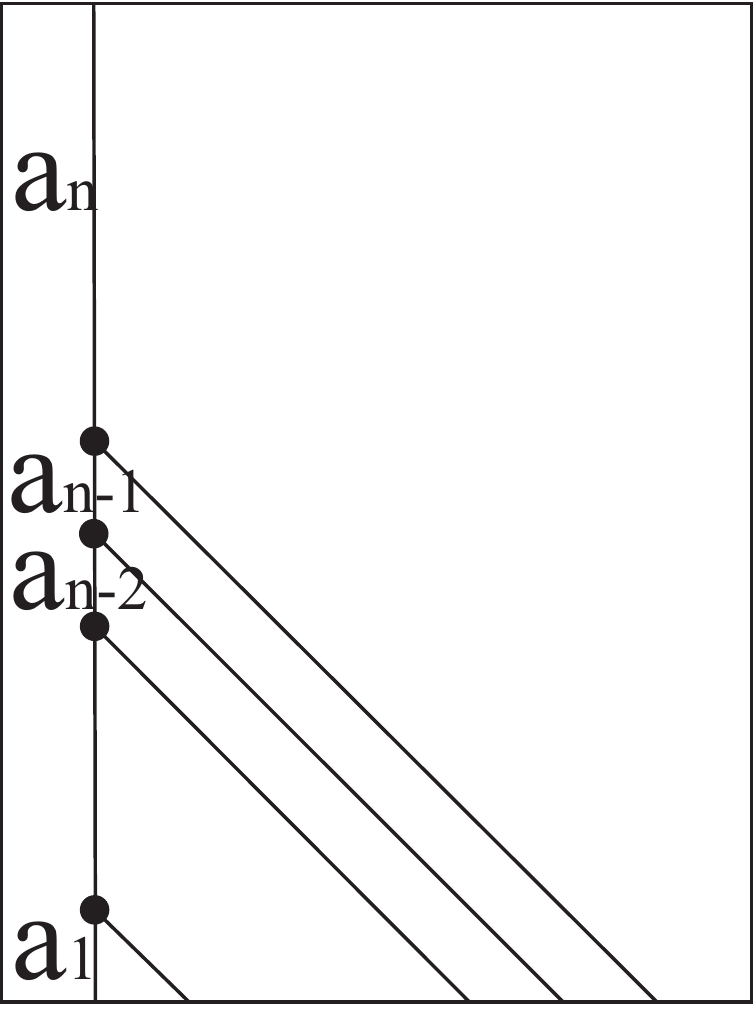}
\caption{The new basis element}
\label{dbasis}
\end{figure}

\begin{rem}
\label{conditions}
It is easy to see that the restrictions on $a_1, \ldots, a_n$ are equivalent to the following conditions:
\begin{itemize}
\item $a_n=h, a_1=1$ and $a_i\geq 0$ for $1<i<n$;
\item $|a_i-a_{i+1}|=1$ for all $0<i<n-1$.
\end{itemize}
\end{rem}

\begin{prop}
\label{basisnorm} 
We have
\[
G(D_{a_1, \ldots, a_n}, D_{b_1, \ldots, b_n})=
\left\{
\begin{array}{l l}
  \prod_{i=1}^n\frac{\theta(a_{i+1},a_i,1)}{\Delta_{a_{i+1}}} & \quad \text{if $(a_1,\ldots,a_n)=(b_1,\ldots,b_n)$};\\
  0 & \quad \text{else}.\\
\end{array} \right.
\]
\end{prop}
\begin{proof}
We just need to repeatedly use the following standard result:
\begin{equation}
\begin{minipage}{0.8in}\includegraphics[width=0.8in]{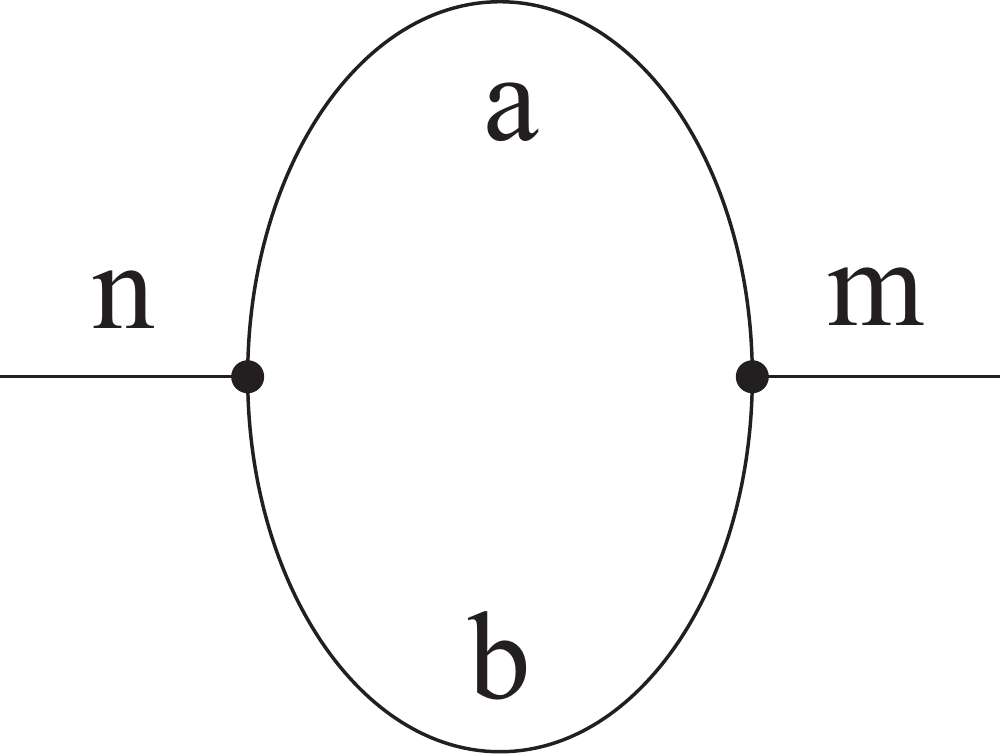}\end{minipage}=\frac{\delta_{n,m}\theta(n,a,b)}{\Delta_n}\
\begin{minipage}{0.8in}\includegraphics[width=0.8in]{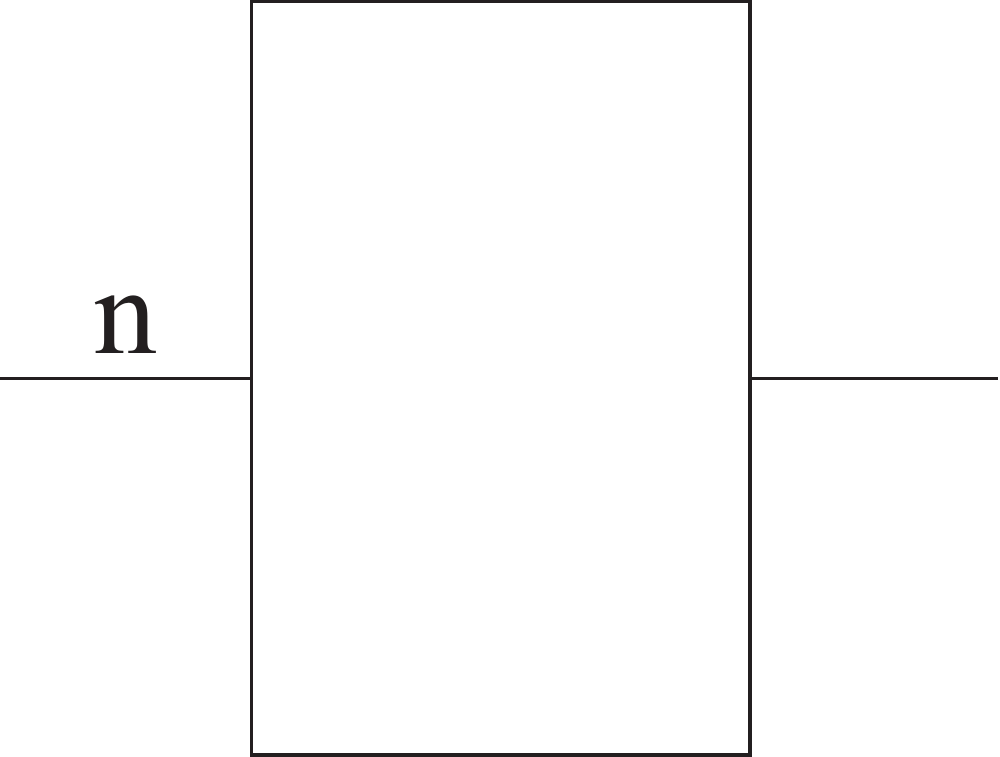}\end{minipage},
\notag
\end{equation}
where $\delta_{n,m}$ is the Kronecker delta.
\end{proof}

\begin{cor}
\label{orthogonal} 
The elements
$D_{a_1, \ldots, a_n}$ are orthogonal with respect to the natural bilinear form.
\end{cor}
\begin{proof}
The proof follows immediately from Proposition \ref{basisnorm}.
\end{proof}

\begin{prop}
\label{generate}
Each element in $\cs(I\times I,n,h)$ can be written as a linear combination of elements in $\cd_n^h$.
\end{prop}
\begin{proof}
It is sufficient to show the result holds for all elements in $\cb_n^h$.
We proceed the proof by induction on $n\geq h$.
Clearly, the result is true for $\cb_h^h$.
Suppose the result is true for $\cb_k^h$ with $h\leq k\leq n-1$.
The proof for the element $\cb_n^h$ is obtained easily from Figure \ref{induction}.
\end{proof}

\begin{figure}[h]
\centering
\includegraphics[width=0.8in, height=1in]{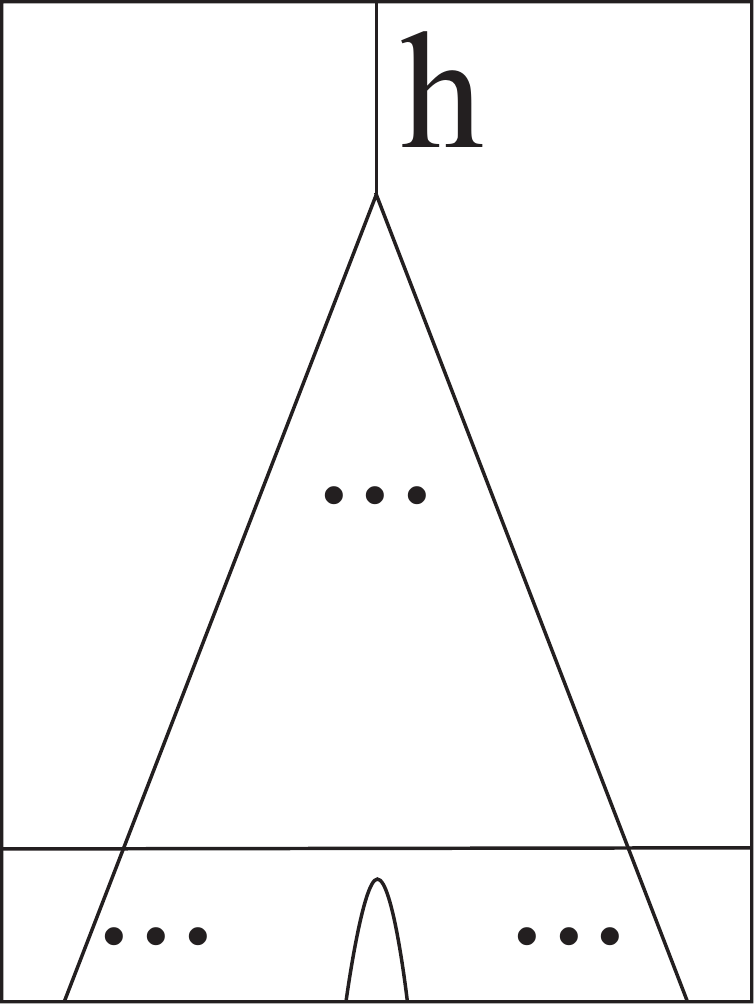}
\includegraphics[width=0.2in, height=1in]{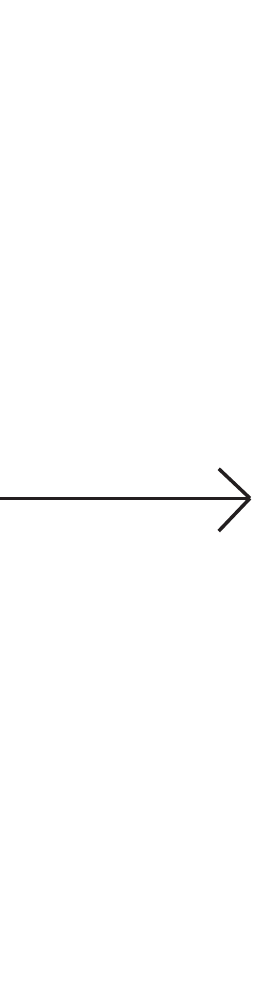}
\includegraphics[width=0.8in, height=1in]{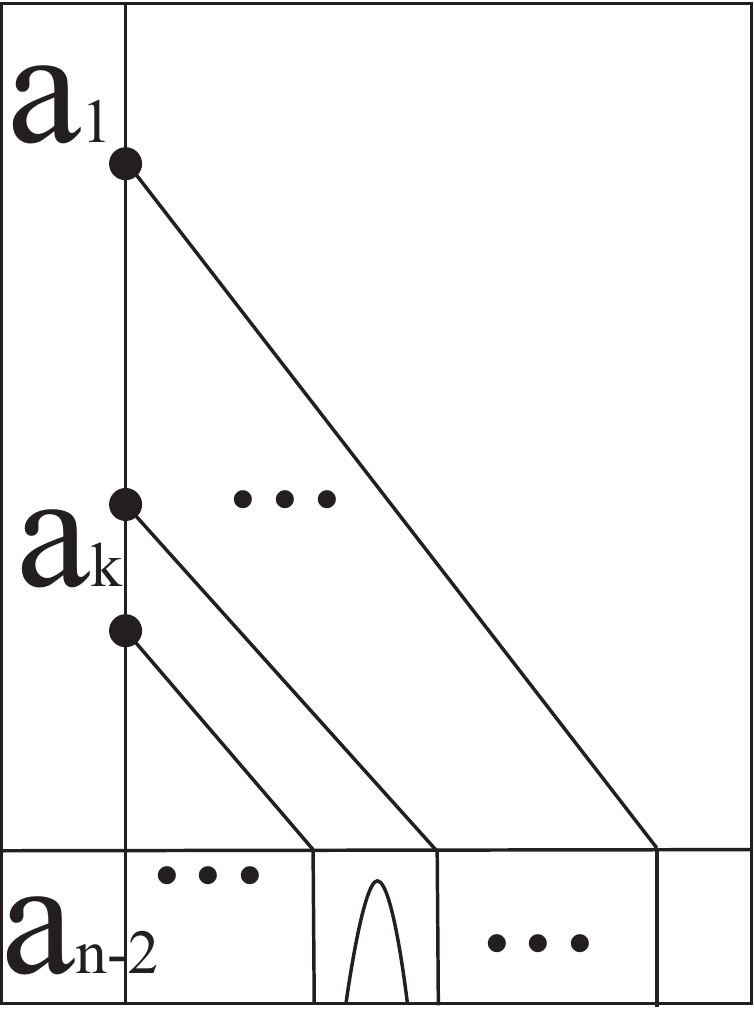}
\includegraphics[width=0.2in, height=1in]{varrow.pdf}
\includegraphics[width=0.8in, height=1in]{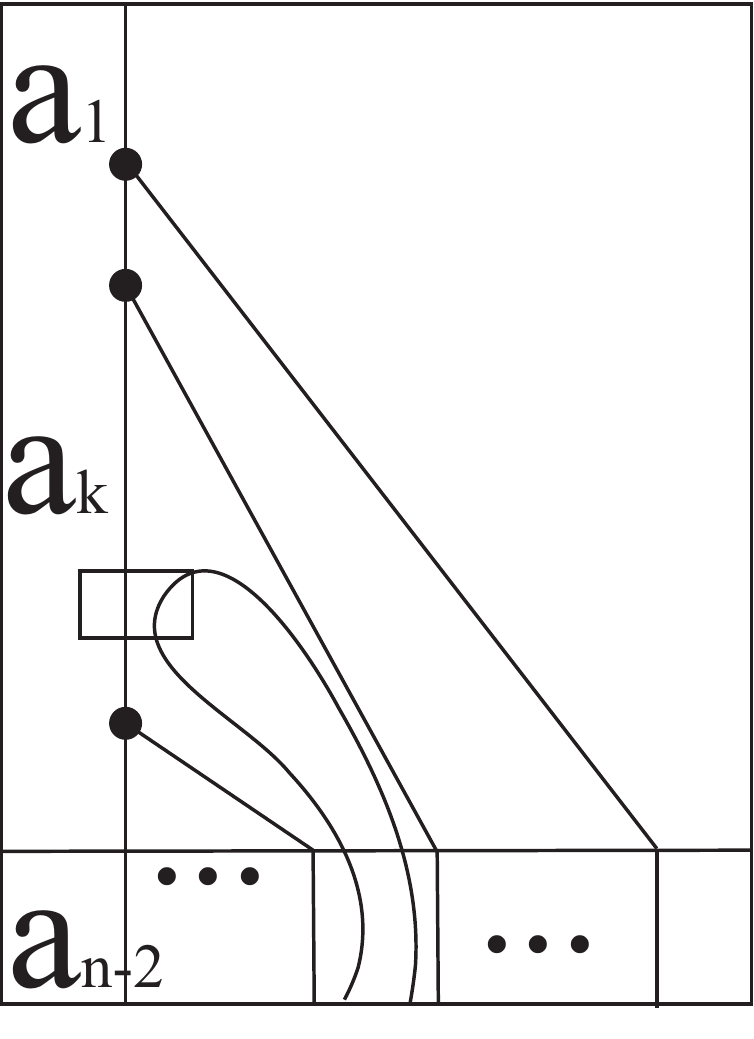}
\includegraphics[width=0.2in, height=1in]{varrow.pdf}
\includegraphics[width=0.8in, height=1in]{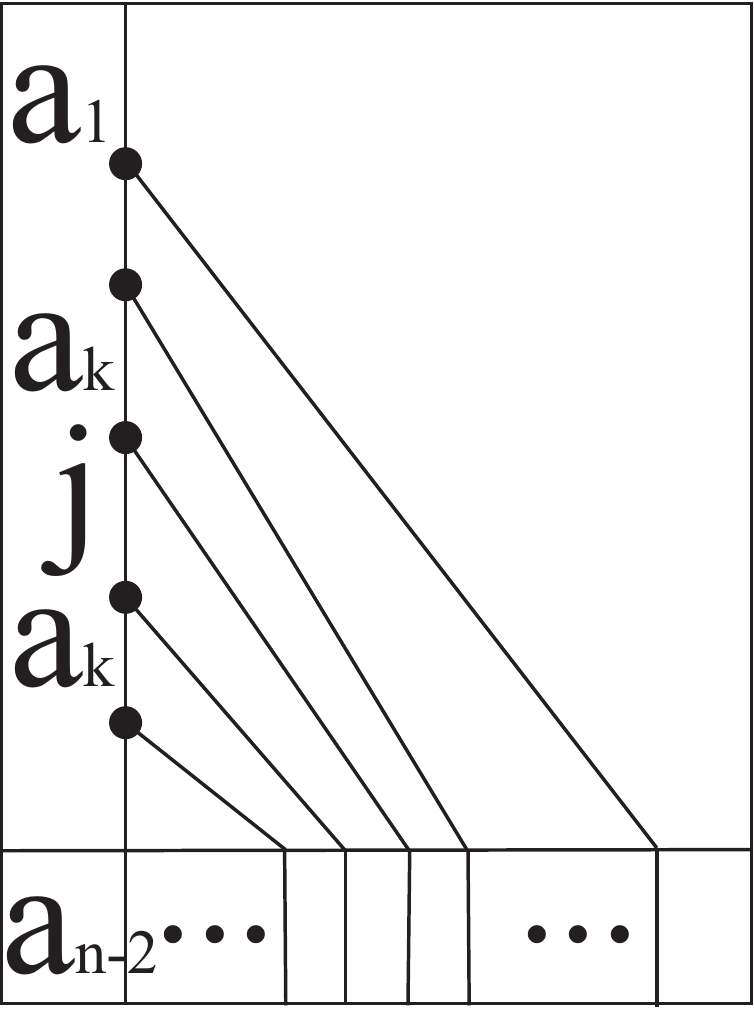}
\caption{For an element $x\in\cb_n^h$, we consider it as being the composition of two layers as in the first Diagram.
The upper half diagram is in $\cb_{n-2}^h$. By induction, we can then write $x$ as a linear combination of figures in the second diagram.
Next, we pull over the turn-back and apply the Fusion Lemma \cite{K2} in skein theory to the boxed area.
We can write $x$ as a linear combination of the figures in the last diagram}
\label{induction}
\end{figure}

\begin{cor}
\label{isabasis}
The set $\cd_n^h$ is a basis for $\cs(I\times I,n,h)$.
\end{cor}
\begin{proof}
The proof follows immediately from Corollary \ref{orthogonal} and Proposition \ref{generate}.
\end{proof}

By the Proposition \ref{basisnorm} and Corollary \ref{isabasis}, 
we obtain the fact that the matrix of the bilinear form with respect to the basis $\cd_n^h$ is diagonal.
To calculate its determinant, we just multiply the diagonal entries together.
However, we want to calculate the determinant of the bilinear form with respect to the natural basis $\cb^h_n$.
Thus, we need to find the transformation matrix $A$ between $\cd_n^h$ and $\cb_n^h$.
To do so, we have to align all the elements in $\cd_n^h$ and $\cb_n^h$.
At first, we define a total order on $\cd_n^h$ and $\cb_n^h$.

\begin{de}
For a set of $n$-tuples $\{(a_1, \ldots, a_n)\}$, we give a lexigraphic order on it as follows:
\begin{itemize}
\item $(a_1, \ldots, a_n)>(b_1, \ldots, b_n)$ if there is a $j$ such that $a_i=b_i$ for all $i<j$ and $a_j>b_j$;
\item $(a_1, \ldots, a_n)=(b_1, \ldots, b_n)$ if $a_i=b_i$ for all $i$;
\item  $(a_1, \ldots, a_n)<(b_1, \ldots, b_n)$, otherwise.
\end{itemize}
\end{de}

Since each element in $\cd_n^h$ corresponds to an $n$-tuple, we can give the $\cd_n^h$ above total order.
We can then align the elements in $\cd_n^h$ from the maximum to the minimum vertically.
We will assign a new system to denote elements in $\cb_n^h$ such that we can do the same thing.

\begin{de}
We construct a new element $B_{a_1,...,a_n}$ of $\cs(I\times I,n,h)$ from $D_{a_1,...,a_n}$ as follows.
We do not insert corresponding idempotents into segments except for $a_n$.
We delete each black dot in the diagram and put a circle around it.
If $a_i=a_{i+1}+1$, then we put the diagram on the left of Figure \ref{circle} in the circle.
If $a_i=a_{i+1}-1$, then we put the diagram on the right of Figure \ref{circle} in the circle.
\end{de}
\begin{figure}[htp]
\includegraphics[width=0.6in, height=0.6in]{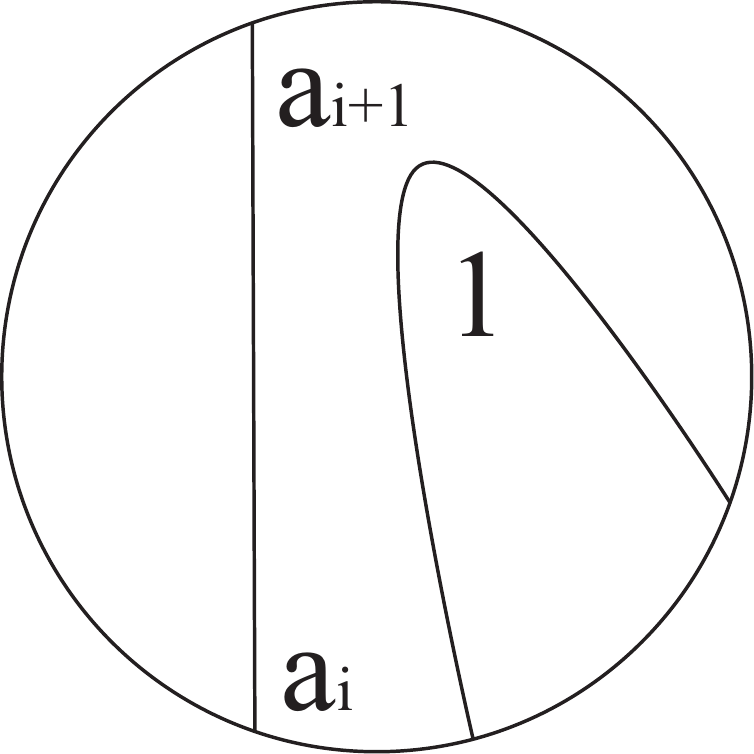}\hskip 0.5in
\includegraphics[width=0.6in, height=0.6in]{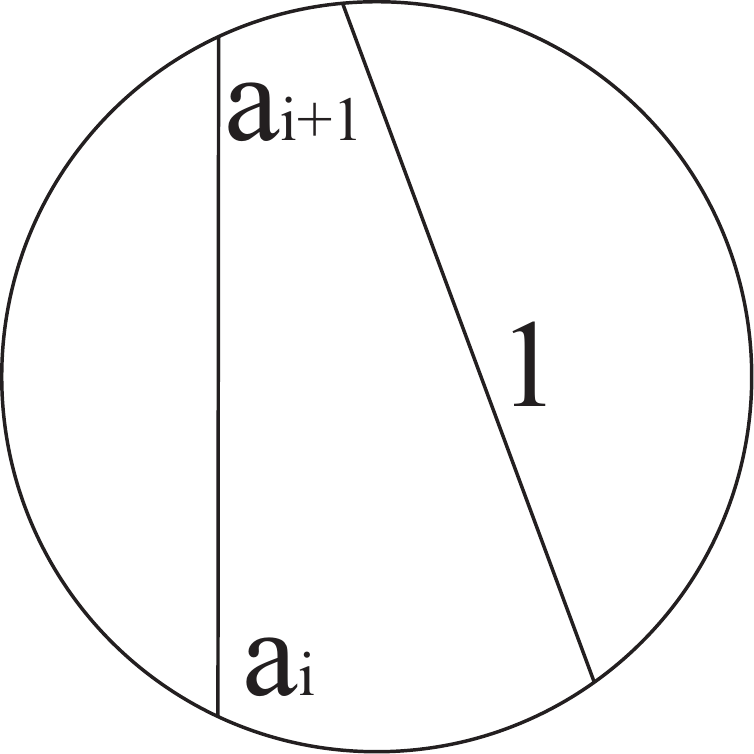}
\caption{We fill the circle with those two diagrams according to $a_i$ and $a_{i+1}$}
\label{circle}
\end{figure}

The proofs of the following two lemmas are very similar to the proofs in \cite[Lemma 5.1, Proposition 5.4]{C}.

\begin{lem}\label{equal} For all $n$,
$$G(D_{a_1,\ldots,a_n},B_{a_1,\ldots,a_n})=G(D_{a_1,\ldots,a_n},D_{a_1,\ldots,a_n}).$$
\end{lem}

\begin{lem}\label{0}
If $(a_1,\ldots,a_n)>(b_1,\ldots,b_n)$, then
$$G(D_{a_1,\ldots,a_n},B_{b_1,\ldots,b_n})=0.$$
\end{lem}

\begin{prop} We have
\begin{center}
$\cb_n^h=\{B_{a_1, \ldots, a_n}\mid (a_1, \ldots, a_n)\ \text{satisfies the conditions in Remark \ref{conditions}}\}$.
\end{center}
\end{prop}
\begin{proof}
It is easy to see that $|\cb_n^h|=|\{B_{a_1, \ldots, a_n}\}|$ and $\{B_{a_1, \ldots, a_n}\}\subset\cb_n^h$.
Therefore, it remains to prove that $B_{a_1, \ldots, a_n}\neq B_{b_1, \ldots, b_n}$ if $(a_1,\ldots,a_n)\neq(b_1,\ldots,b_n)$.
But this follows from the fact that
\begin{align*}
&\biggl(G(B_{a_1,\ldots,a_n},D_{a_1,\ldots,a_n}),G(B_{a_1,\ldots,a_n},D_{b_1,\ldots,b_n})\biggr)\\
&\qquad\qquad\neq\biggl(G(B_{b_1,\ldots,b_n},D_{a_1,\ldots,a_n}),G(B_{b_1,\ldots,b_n},D_{b_1,\ldots,b_n})\biggr)
\end{align*}
by Lemmas \ref{equal} and \ref{0}.
\end{proof}

Now, we can correspond to each element in $\cb_n^h$ a $n$-tuple $(a_1,...,a_n)$ and give $\cb_n^h$ the same total order as we did on $\cd_n^h$.
We align the elements in $\cb_n^h$ from maximum to minimum.
Then we can write $B$'s in term of $D$'s as follows:
\[
	\begin{pmatrix}
	B_{1,2,\ldots,h,h-1,h,h-1,\ldots,h}\\
	\vdots\\
	B_{1,0,1,0,\ldots,1,2,3,\ldots,h-1,h}
	\end{pmatrix}
	=
	A
 	\begin{pmatrix}
	D_{1,2,\ldots,h,h-1,h,h-1,\ldots,h}\\
	\vdots\\
	D_{1,0,1,0,\ldots1,2,3\ldots,h-1,h}
	\end{pmatrix},
\]
where $A$ is the transform matrix.

\begin{prop}
\label{transformmatrix}
We have
$$A =
\begin{pmatrix}
1 & * & \cdots &* \\
0 & 1 & \cdots &* \\
\vdots  & \vdots  & \ddots & \vdots  \\
0 & 0 & \cdots & 1
\end{pmatrix}.$$
\end{prop}
\begin{proof}
The proof follows from Lemma \ref{equal} and Lemma \ref{0}.
\end{proof}

\begin{cor}
$\cb_n^h$ is a basis for $\cs(I\times I,n,h)$.
\end{cor}
\begin{proof}
As we can see in Corollary \ref{isabasis}, $\cd_n^h$ is a basis for $\cs(I\times I,n,h)$.
By Proposition \ref{transformmatrix}, the transformation matrix between $\cb_n^h$ and $\cd_n^h$ is nondegenerate.
Therefore, elements in $\cb_n^h$ are linearly independent.
Moreover, by Proposition \ref{generating}, $\cb_n^h$ is a generating set.
Thus, we have that $\cb_n^h$ is a basis of $\cs(I\times I,n,h)$.
\end{proof}

\section{Proof of Main Result}\label{PMR}
Now we ready to prove our main result, Theorem \ref{main}.
We denote the matrix of $G$ with respect to $\cd^h_n$ by $D$. Proposition \ref{transformmatrix} gives $\det(B)=\det(D)$.
By Corollary \ref{orthogonal} we have $\{D_{a_1,\ldots,a_{2n-1}}\}$ is an orthogonal basis with respect to the bilinear form.
Thus $D$ is a diagonal matrix with respect to this basis. Therefore,
\begin{align*}
\det(D)=\prod_{(a_1,\ldots,a_n)}\langle D_{a_1,\ldots,a_n},D_{a_1,\ldots,a_n}\rangle.
\end{align*}
Thus, by Proposition \ref{basisnorm} we obtain
\begin{align}\label{eqDD1}
\det(D)=\prod_{(a_1,\ldots,a_n)}(\Delta_{a_1}\prod_i\frac{\theta(a_{i+1},a_i,1)}{\Delta_{a_{i+1}}}).
\end{align}
In order to simplify the expression $\det(D)$, we need the following lemma which holds immediately from the definitions.

\begin{lem} We have
$$\frac{\theta(a_{i+1},a_i,1)}{\Delta_{a_{i+1}}}=
\left\{
\begin{array}{l l}
  \frac{\Delta_{a_i}}{\Delta_{a_{i+1}}} & \quad \text{if $a_{i+1}=a_i-1$};\\
  1 & \quad \text{if $a_{i+1}=a_i+1$}.\\
\end{array} \right.$$
\label{simplify}
\end{lem}

Lemma \ref{simplify} and \eqref{eqDD1} give
\begin{align*}
\det(D)=\Delta_h^{|\cd_n^h|}\prod_{a_1,\ldots,a_n}(\frac{\Delta_k}{\Delta_{k-1}})^{\alpha^k_{(n,h)}},
\end{align*}
where $\alpha^k_{(n,h)}$ is the number of times that $\frac{\theta(a_{i+1},a_i,1)}{\Delta_{a_{i+1}}}=\frac{\Delta_k}{\Delta_{k-1}}$.
Hence, our problem is reduced to count the number of all such $\frac{\theta(a_{i+1},a_i,1)}{\Delta_{a_{i+1}}}$'s, for a fixed $k$.
To do so, we need the combinatorial structure.

\begin{de}
A lattice path in the plane is a path from $(0,0)$ to $(n,h)$ with northeast and southeast unit steps,
where $n\in\bn,$ and $h\in\bz$.
A generalized Dyck path is a lattice path that never goes below the $x$-axis. 
We denote the set of all generalized Dyck paths from $(0,0)$ to $(n,h)$ by $\cd_{(n,h)}$. We define $a$-shifted generalized Dyck path to be generalized Dyck path $D$ such that we map each point $(x,y)$ of $D$ to $(x+a,y+a)$.
A Dyck path is a generalized Dyck path from $(0,0)$ to $(n,0)$.
\end{de}

\begin{figure}[htp]
\centering
\includegraphics[width=0.8in,height=1in]{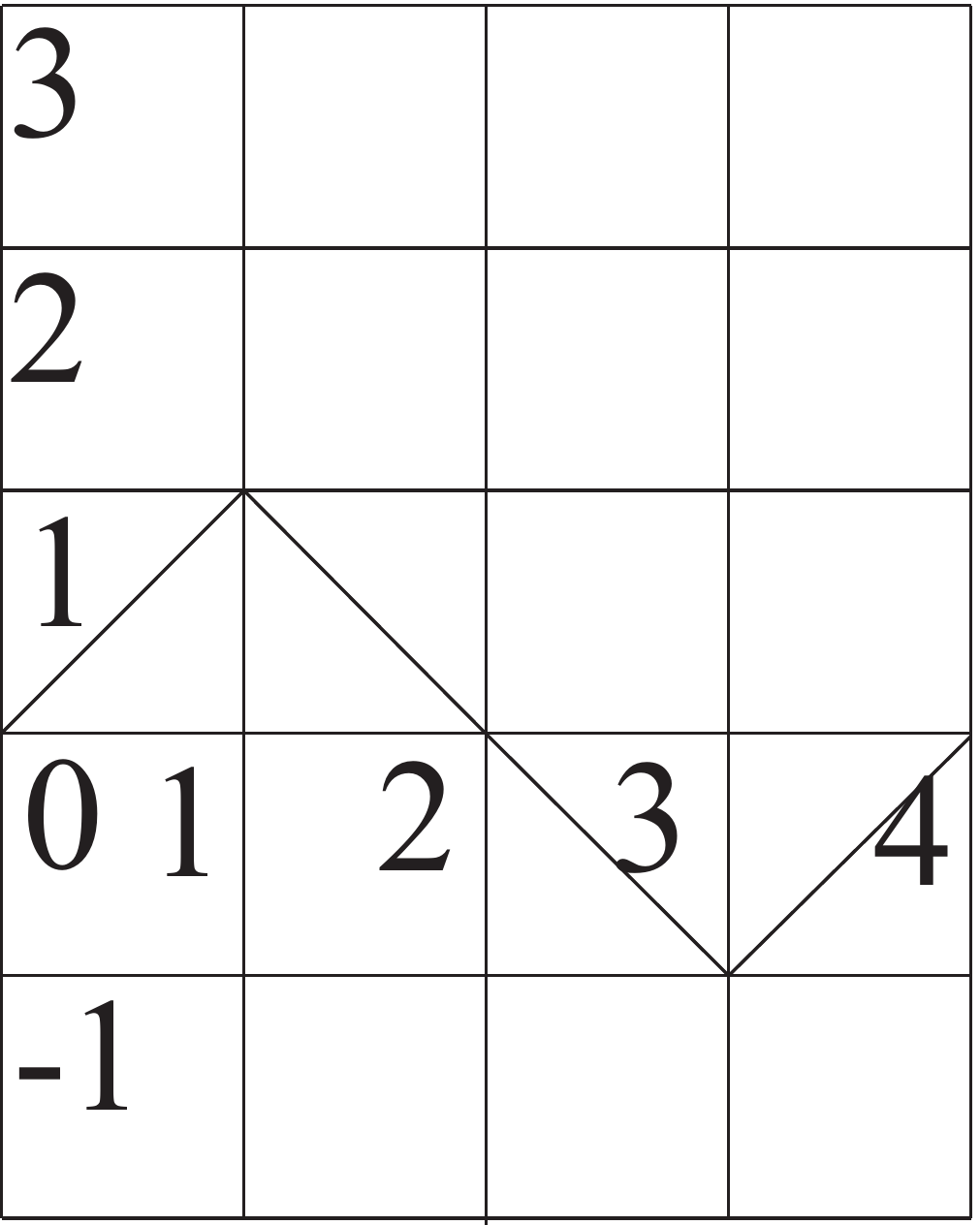}\hskip 0.5in
\includegraphics[width=0.8in,height=1in]{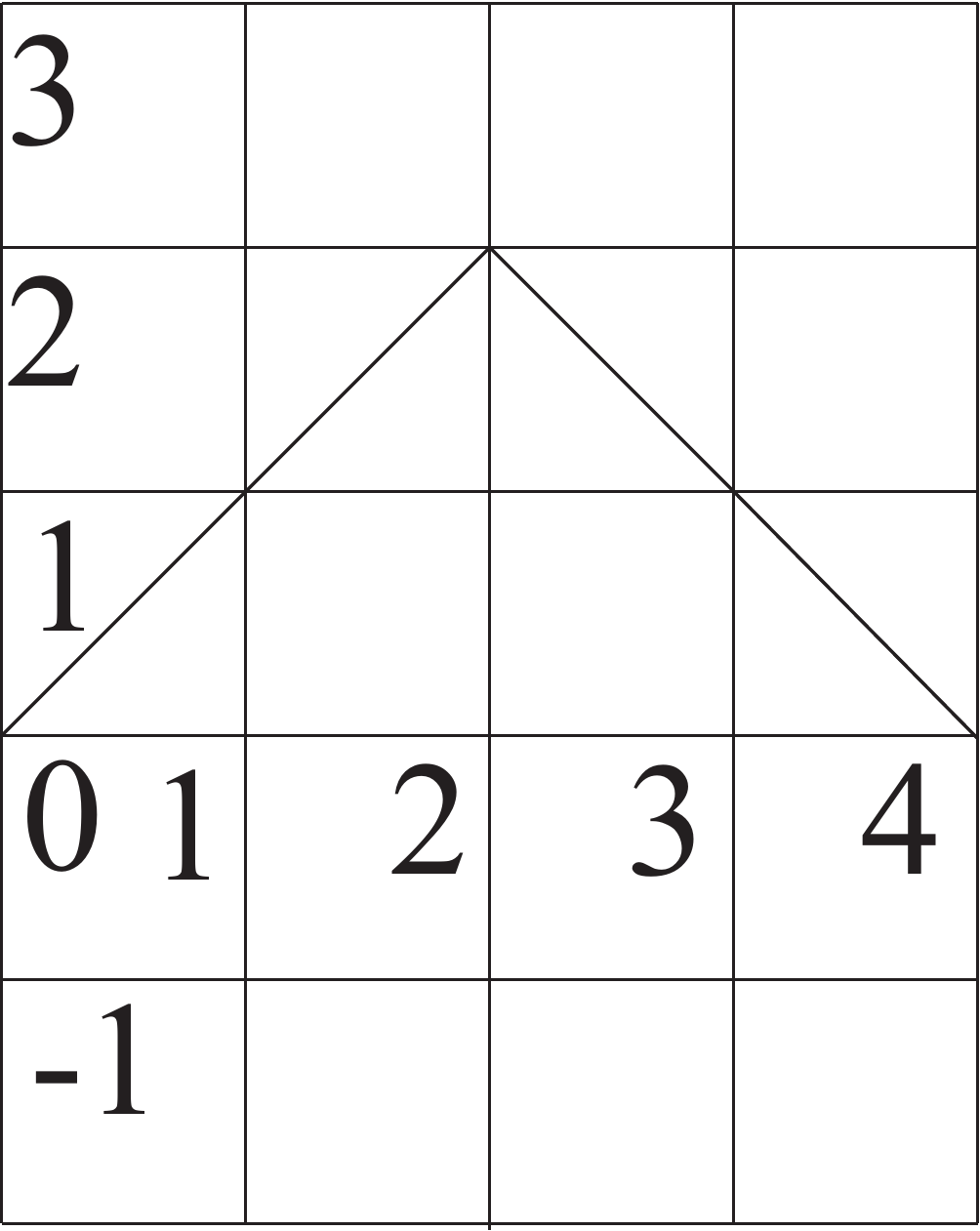}
\caption{On the left is a lattice path from $(0,0)$ to $(4,0)$, and on the right is a generalized Dyck path from $(0,0)$ to $(4,0)$}
\label{path}
\end{figure}

As we can see there is a $1-1$ correspondence between $\cd_n^h$ and $n$-tuples $\{(a_1,\ldots,a_n)\}$ satisfying the conditions in Remark \ref{conditions}.
Note that there is a $1-1$ correspondence between the $n$-tuples $\{(a_1,\ldots,a_n)\}$ and $(n+1)$-tuples $\{(0,a_1,\ldots,a_n)\}$.
Therefore, there is a $1-1$ correspondence between the $(n+1)$-tuples $\{(0,a_1,\ldots,a_n)\}$ and $\cd_{(n,h)}$,
for any $(a_1,\ldots,a_n)$ satisfying the conditions in Remark \ref{conditions}.
Hence, there is a $1-1$ correspondence between the sets $\cd_n^h$ and $\cd_{(n,h)}$,
that is, $|\cd_n^h|=|\cd_{(n,h)}|$.

\begin{de}
A $k$-down step in a generalized Dyck path is a southeast step from height $k$ to height $k-1$, see Figure \ref{downstep}.
\end{de}

\begin{figure}[h]
\centering
\includegraphics[width=1in,height=1in]{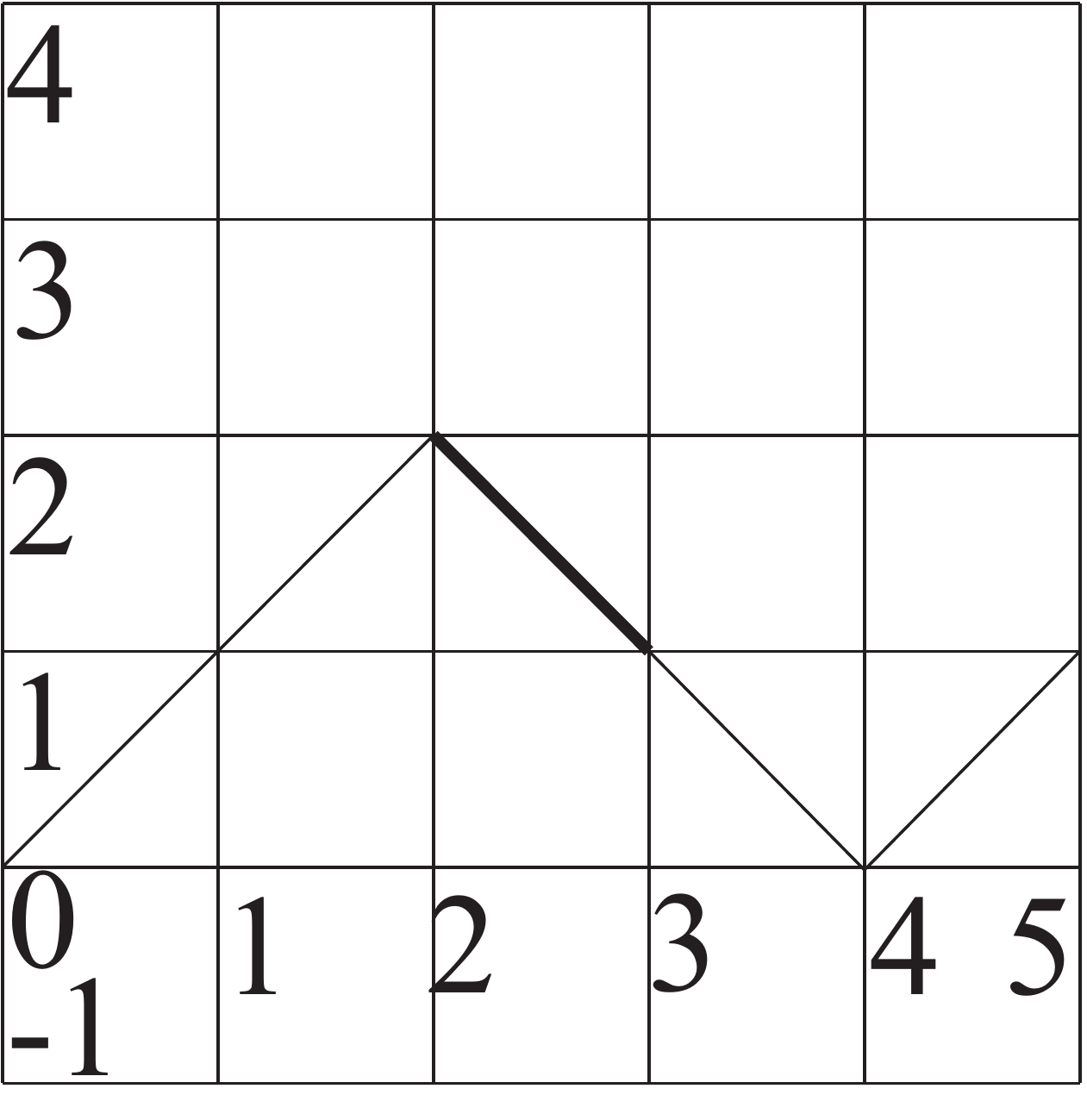}
\caption{A 2-down step}
\label{downstep}
\end{figure}

Then the problem of counting all $\frac{\theta(a_{i+1},a_i,1)}{\Delta_{a_{i+1}}}=\frac{\Delta_k}{\Delta_{k-1}}$ in $\{(a_1,\ldots,a_n)\}$ is equivalent to count all pairs $(D,i)$, where $D\in\cd_{(n,h)}$ and $a_i=k, a_{i+1}=k-1$.
We denote set of these pairs by $A_{(n,h)}^k$.
Geometrically, each pair $(D,i)$ corresponds to a $k$-down step in the generalized Dyck path $D$.
Counting all pair $(D,i)$ is the same as counting all $k$-down step in all generalized Dyck path in $\cd_{(n,h)}$.
The aim of the next section is to count all this kind of steps.

\section{A combinatorial result}\label{ACR}
In this section, we are going to prove the following result.

\begin{thm}\label{combinatorialresult}
For all $n$,
\begin{equation}
\left|A_{(n,h)}^k\right|=\alpha^k_{(n,h)}=\binom{n}{\frac{n+h+2k-2s}{2}}-\binom{n}{\frac{n+h+2k}{2}+1},
\notag
\end{equation}
where $s=\min\{k-1,h\}$.
\end{thm}

We prove this theorem by using two different approaches.
In Section \ref{subsection1}, we present a combinatorial and geometric explanation,
which was inspired by Di Francesco's proof for the case $h=0$, see \cite[Proposition 2]{F}.
In Section \ref{subsection2},
we present our second approach which is based on the generating function techniques,
which provides an alternative proof of Di Francesco's Proposition 2 in \cite{F}.

\subsection{A geometric proof}\label{subsection1}
In this section, we will construct two maps:
\begin{equation}
\Theta:A^{k}_{(n,h)}\rightarrow\bigcup_{j=0}^{s}\cd_{(n,2k-2j+h)}\label{theta}
\end{equation}
and
\begin{equation}
\Phi:\bigcup_{j=0}^{s}\cd_{(n,2k-2j+h)}\rightarrow A^{k}_{(n,h)},\label{phi}
\end{equation}
where $s=\min\{k-1,h\}$.
We will prove that $\Phi\Theta=id$ and $\Theta\Phi=id$. 
Thus, both of them are bijective.
By the reflection principle, we know that
\begin{equation}
|{\cd_{(n,2k-2j+h)}}|=\binom{n}{\frac{n+2k-2j+h}{2}}-\binom{n}{\frac{n+2k-2j+h}{2}+1}.
\label{cardofdyckpath}
\end{equation}
Therefore,
\begin{align*}
\left|\bigcup_{j=0}^{s}\cd_{(n,2k-2j+h)}\right|&=\sum_{j=0}^{s}{\binom{n}{\frac{n+2k-2j+h}{2}}-\binom{n}{\frac{n+2k-2j+h}{2}+1}}\\
&=\binom{n}{\frac{n+h+2k-2s}{2}}-\binom{n}{\frac{n+h+2k}{2}+1}.
\end{align*}
Then we have
\begin{equation}
\left|A_{(n,h)}^k\right|=\binom{n}{\frac{n+h+2k-2s}{2}}-\binom{n}{\frac{n+h+2k}{2}+1}.
\end{equation}

\begin{figure}[htp]
\centering
\includegraphics[width=1in,height=1.2in]{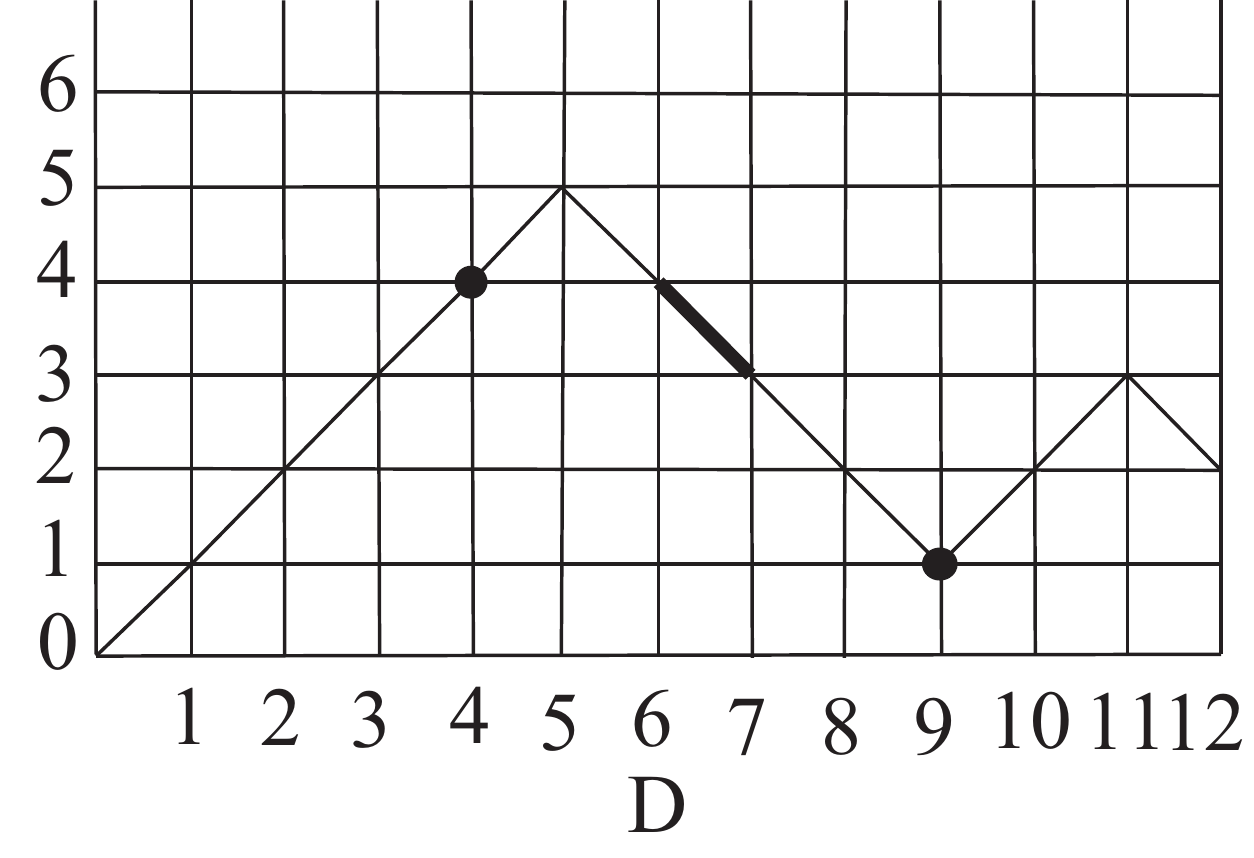}
\includegraphics[width=0.5in,height=1.2in]{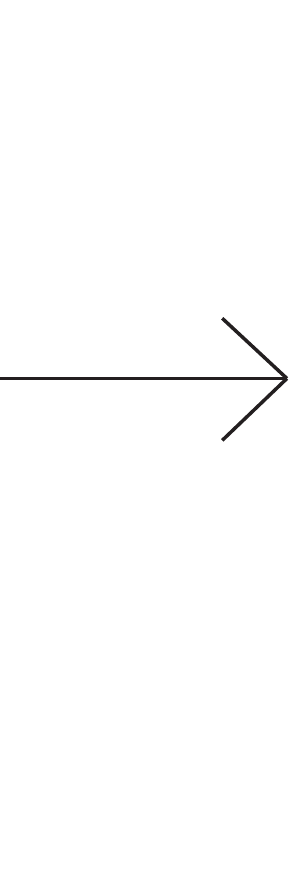}
\includegraphics[width=1.2in,height=1.2in]{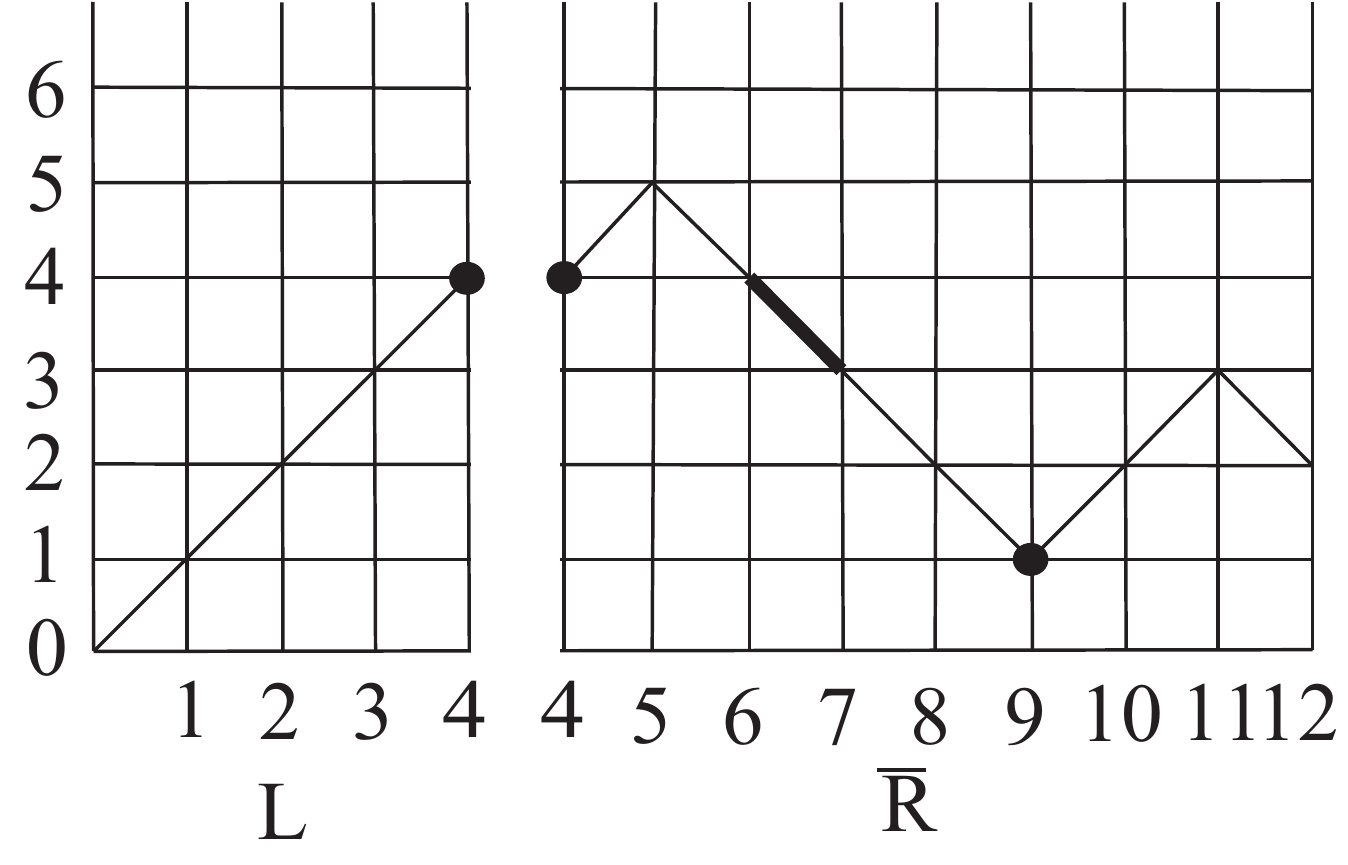}
\includegraphics[width=0.5in, height=1.2in]{arrow.pdf}
\includegraphics[width=1.4in,height=1.2in]{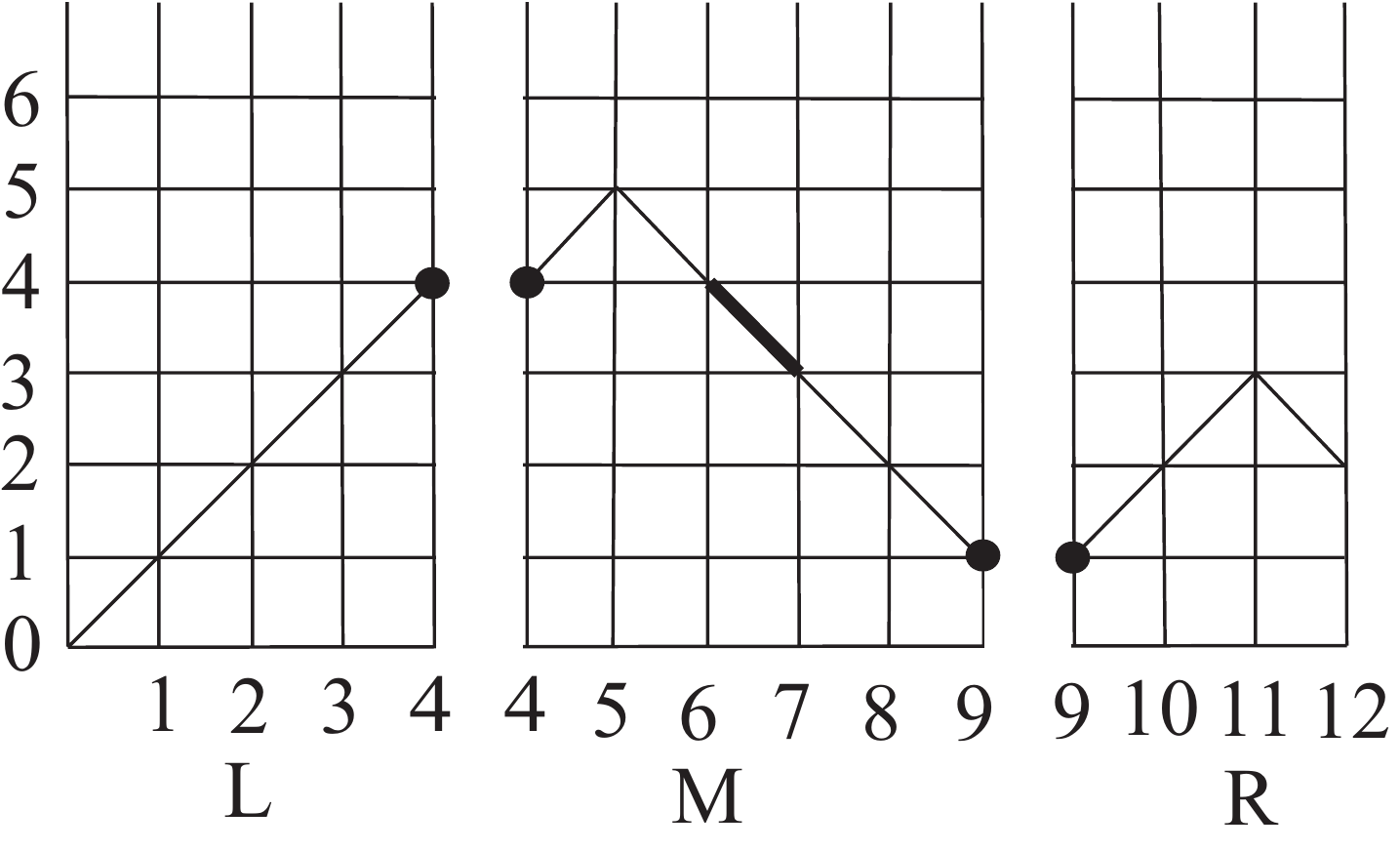}
\caption{The cutting process, where $k=4, n=12, h=2, i=6, l=9$. The bold step is the one we are considering. Since $a_5=5$, we choose $i'=4$. We cut the path at the 4th place into $L$ and $\bar{R}$. Now, the rightmost lowest point in $\bar{R}$ is $(9,1)$, so we cut $\bar{R}$ into two parts $M$ and $R$ at $(9,1)$}\label{cutting}
\end{figure}

{\em\bf Step 1: Construct $\Theta$.} 
Suppose we have a $k$-down step occurring at the $i$th place in the Dyck path $D$. 
We denote this $k$-down step by $(D,i)$.
We cut $D$ into 3 parts as follows (see Figure \ref{cutting}):
\begin{enumerate}
\item We choose the largest $i'\leq i$ such that $a_{i'-1}=k-1$ and $a_{i'}=k$,
we cut the path at $(i',a_{i'})$. We denote the left part by $L$,
the right part by $\bar{R}$.
\item Now we consider the right part $\bar{R}$.
Suppose the lowest height of $\bar{R}$ is $j$ and $(l,j)$ is the lefttmost lowest point of $\bar{R}$.
Then we cut $\bar{R}$ at $(l,j)$ into 2 parts. We denote
the left part of $\bar{R}$ by $M$ and the right part of $\bar{R}$  by $R$.
One may check that $0\leq j\leq s$.
\end{enumerate}
Now the path can be considered as a union of 3 parts $L,M$ and $R$.
We do some operations on $L,M,R$ and glue them back as follows (see Figure \ref{glueback}):
\begin{enumerate}
\item We reflect $R$ with respect to $y$-axis and shift it down by $h-j$ units.
We denote the resulting part by $rR$.
The lowest point of $rR$ is $(n,2j-h)$.
Then by gluing the starting point of $rR$ to the endpoint of $M$,
we get two parts again, $L$ and $M\cup rR$.
It is easy to see that the lowest point of $M\cup rR$ is $(n,2j-h)$.
\item Now we reflect $M\cup rR$ with respect to the $y$-axis and shift it up by $k-2j+h$ units. Then we glue it back to the end of $L$.
\end{enumerate}
\begin{figure}[htp]
\centering
\includegraphics[width=1.4in, height=1.2in]{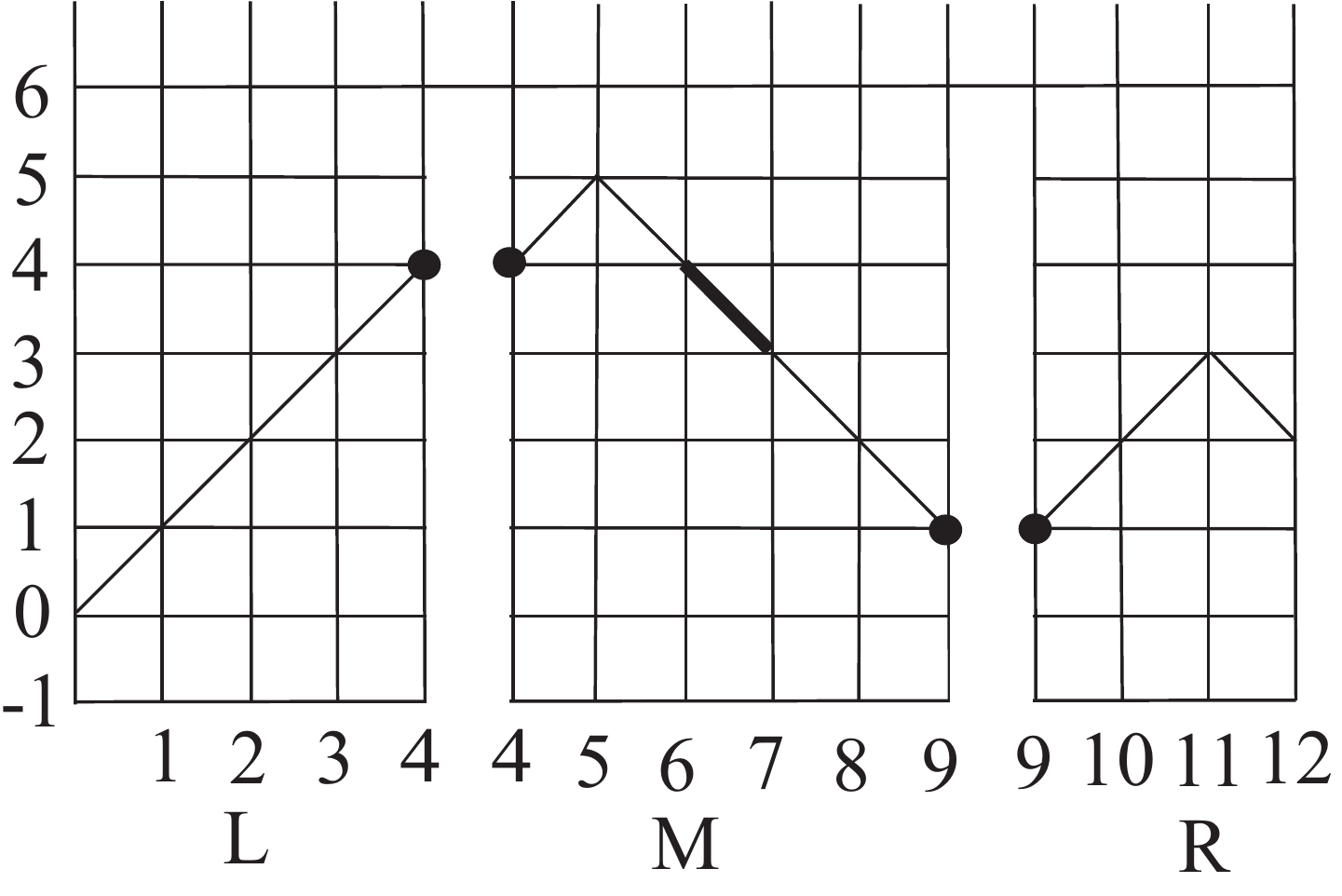}
\includegraphics[width=0.5in, height=1.2in]{arrow.pdf}
\includegraphics[width=1.2in, height=1.2in]{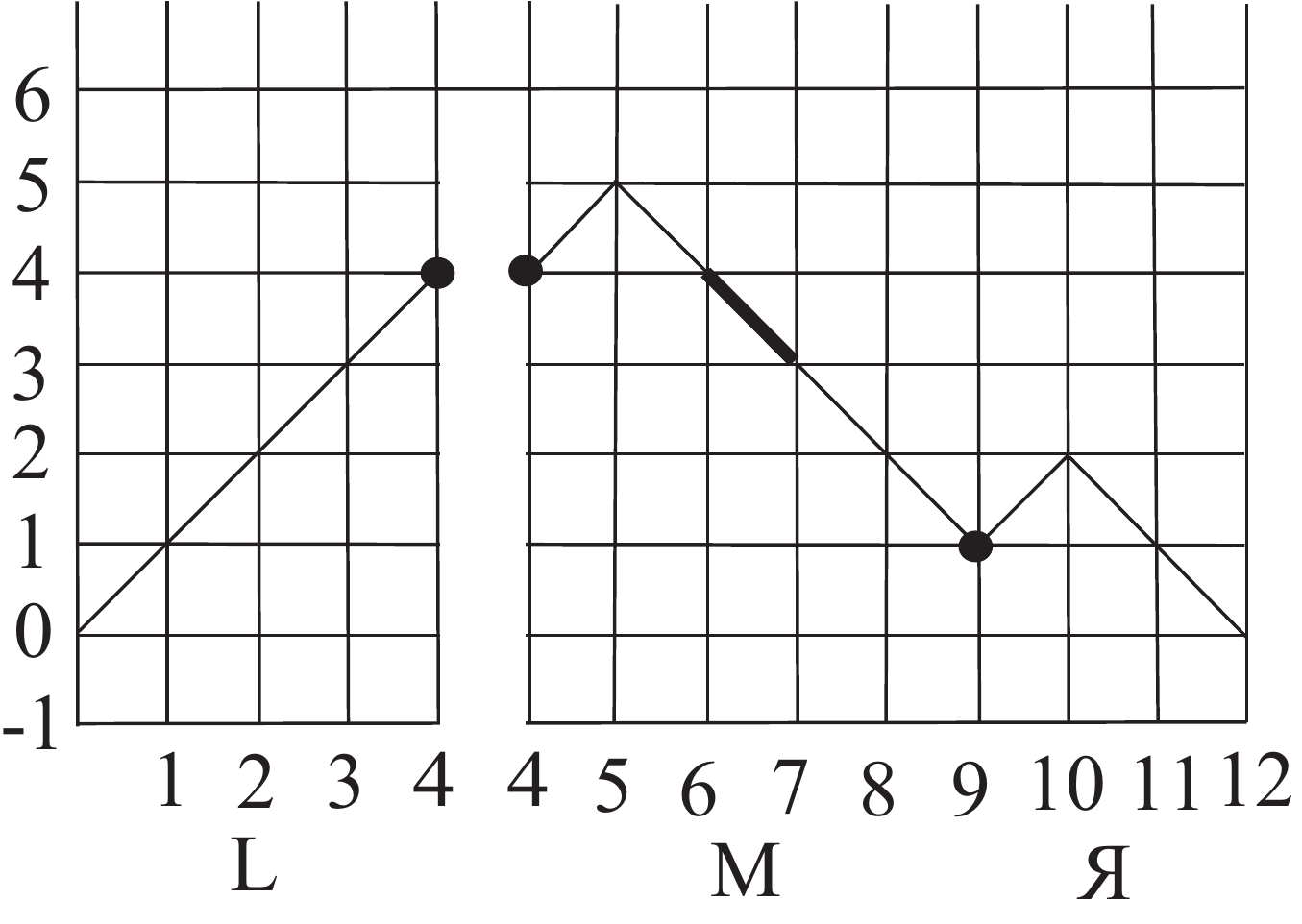}
\includegraphics[width=0.5in, height=1.2in]{arrow.pdf}
\includegraphics[width=1in, height=1.2in]{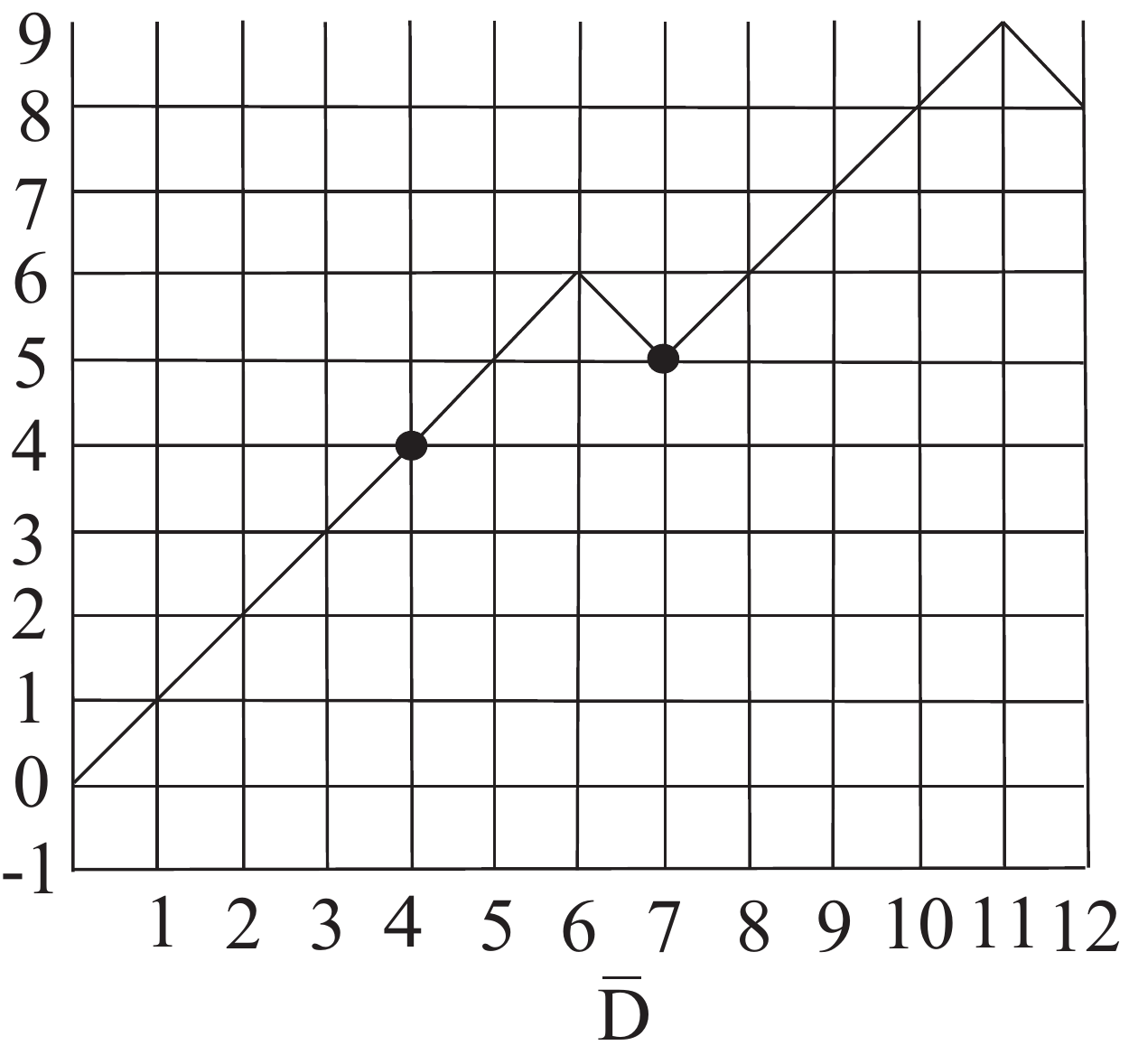}
\caption{The glue-back process. We reflect and shift $R$,  we glue it back to $M$, then we reflect and shift $M\cup rR$ and glue it back to $L$}
\label{glueback}
\end{figure}
At the end, it is easy to see that we have a path $\bar{D}\in\cd_{(n,2k-2j+h)}$.
We set $\Theta((D,i))=\bar{D}$, therefore we have established a map
\begin{equation}
\Theta:A^{k}_{(n,h)}\rightarrow\bigcup_{j=0}^{s}\cd_{(n,2k-2j+h)}.\notag
\end{equation}

{\em\bf Step 2: Construct $\Phi$.}
Basically, this is the reverse process of $\Theta$.
Readers can go through Figure \ref{cutting} and Figure \ref{glueback} backward.
Suppose we have a path $\bar{D}\in\bigcup_{j=0}^{s}\cd_{(n,2k-2j+h)}$.
The whole process is as follows:
\begin{enumerate}
\item Since the end of $\bar{D}$ is $2k-2j+h>k$, where $0\leq j\leq s$,
there is at least one $t$ such that $a_{t-1}=k-1, a_t=k$ and $a_{t+1}=k+1$.
We choose the largest such $t$, denoted by $m$, and cut the path $\bar{D}$ at $(m,a_m)$ into two parts $L$ and $r\bar{R}$.
\item By reflecting $r\bar{R}$ and shift it down by $k-2j+h$ units,
we obtain a new part $\bar{R}$.
The endpoint of $\bar{R}$ is $(n,2j-h)$.
\item It is easy to see that $y=j$ must intersect $\bar{R}$,
since $j$ is always between $k$ and $2j-h$.
We choose the leftmost point of $\{y=j\}\cap\bar{R}$ and denote it by $(v,a_v)$.
We cut $\bar{R}$ into 2 parts $M$ and $rR$ at $(v,a_v)$.
\item We reflect $rR$ with respect to the $y$-axis and shift it up by $h-j$ units. We denote the resulting part by $R$.
At the end, we glue the starting point of $M$ to the endpoint of $L$ and glue the beginning point of $R$ to the endpoint of $M$.
\end{enumerate}
Now we get a Dyck path $D$ from $(0,0)$ to $(n,h)$.
We choose the smallest $m'>m$ such that $a_{m'}=k$ and $a_{m'+1}=k-1$.
We set $\Phi(\bar{D})=(D,m')$. Thus, we have established a map
\begin{equation}
\Phi:\bigcup_{j=0}^{s}\cd_{(n,2k-2j+h)}\rightarrow A^{k}_{(n,h)}.\notag
\end{equation}

\begin{prop}
We have
$\Phi\Theta=id$ and $\Theta\Phi=id$.
\end{prop}
\begin{proof}
Here, we give the detailed proof of the first part of the proposition. 
The second part is similar.
We need to verify that
\begin{equation}
\Phi\Theta:A^{k}_{(n,h)}\rightarrow A^{k}_{(n,h)}
\label{verification}
\end{equation}
is the identity on $A^{k}_{(n,h)}$.
Let us choose an element $(D,i)\in A^{k}_{(n,h)}$, that is, the $k$-down step happens at the $i$th place in the Dyck path $D\in\cd_{(n,h)}$.
We now apply the process of $\Theta$ to $(D,i)$, we get $\bar{D}$, 
see Figure \ref{verification}.
\begin{figure}[h]
\centering
\includegraphics[width=1in, height=1.2in]{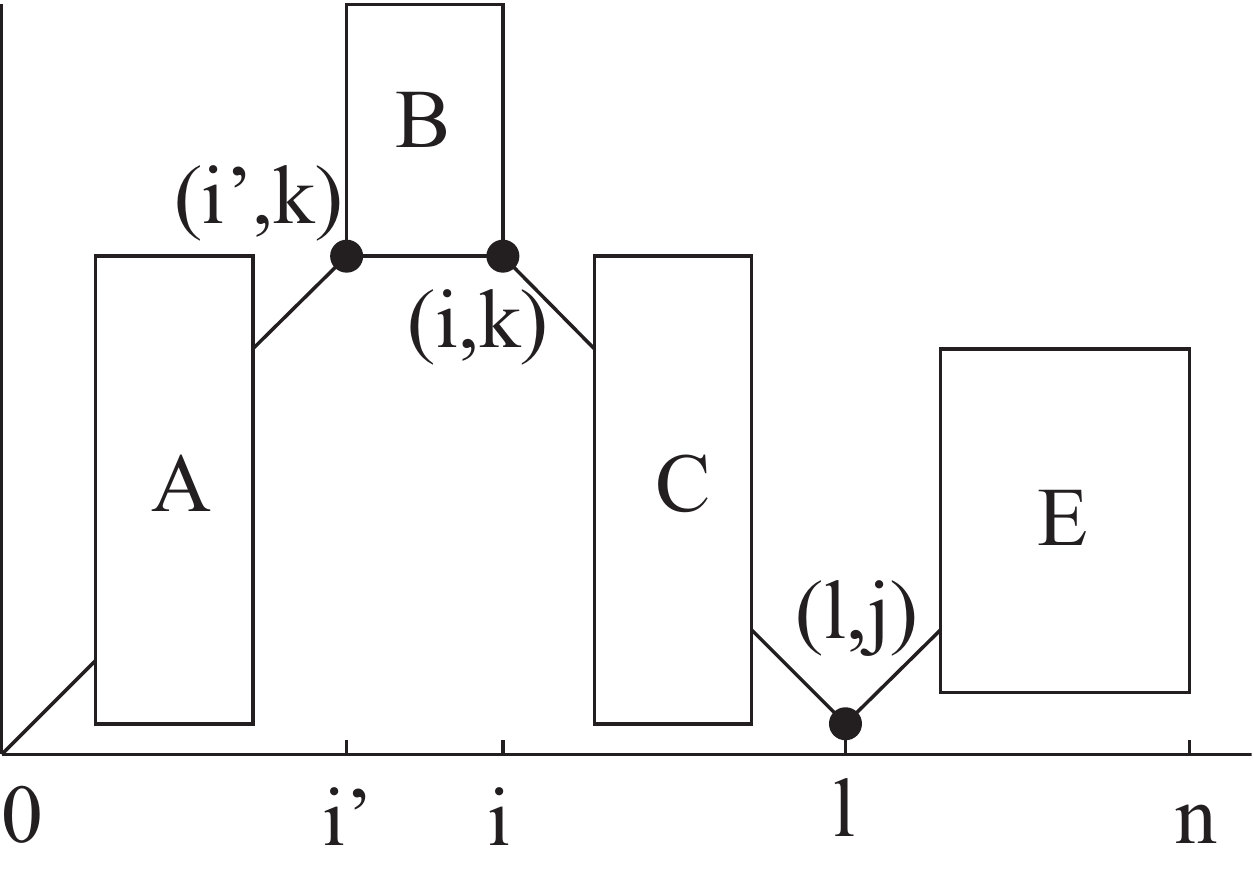}
\includegraphics[width=0.2in, height=1.2in]{arrow.pdf}
\includegraphics[width=1in, height=1.2in]{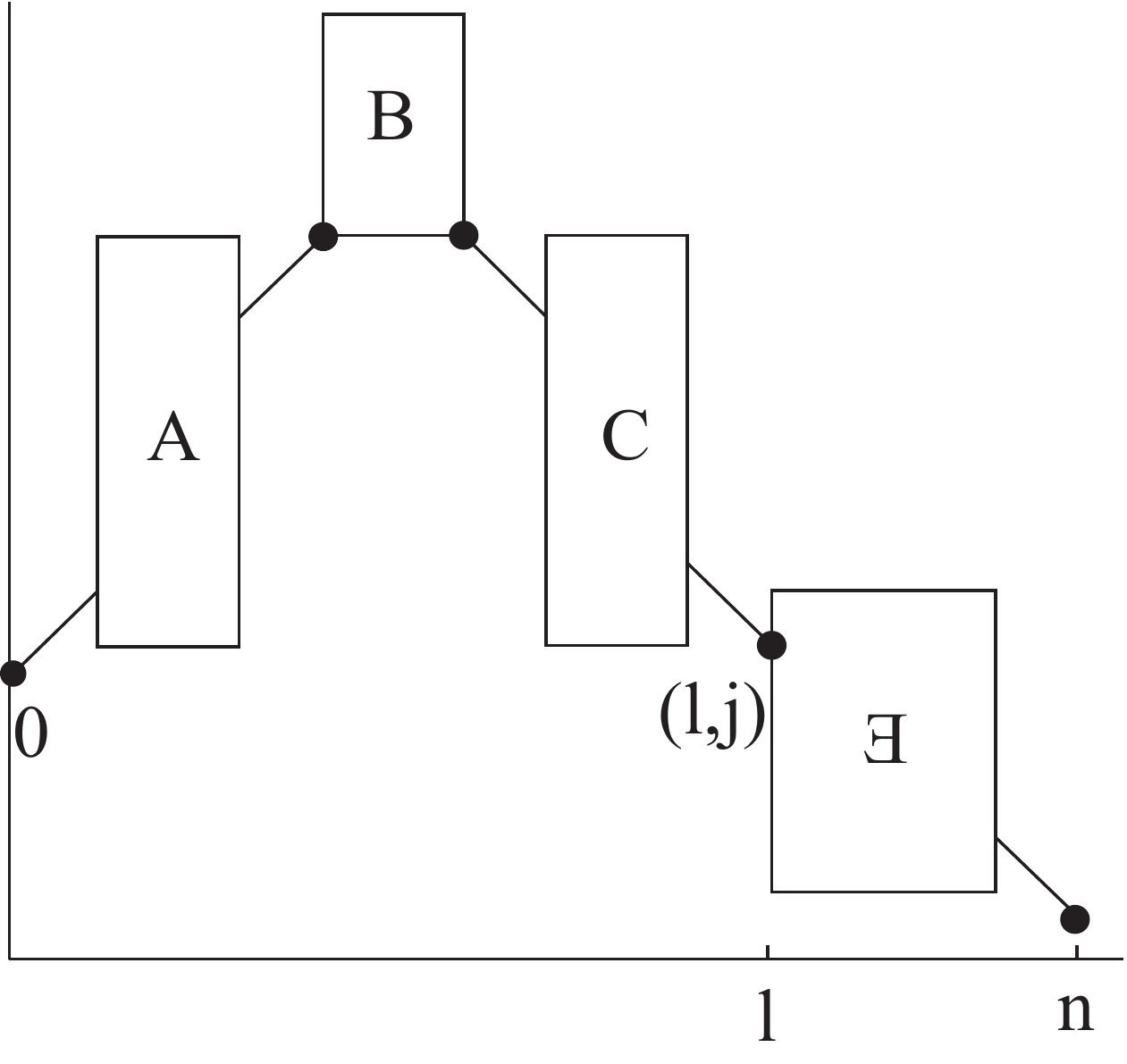}
\includegraphics[width=0.2in, height=1.2in]{arrow.pdf}
\includegraphics[width=1in, height=1.2in]{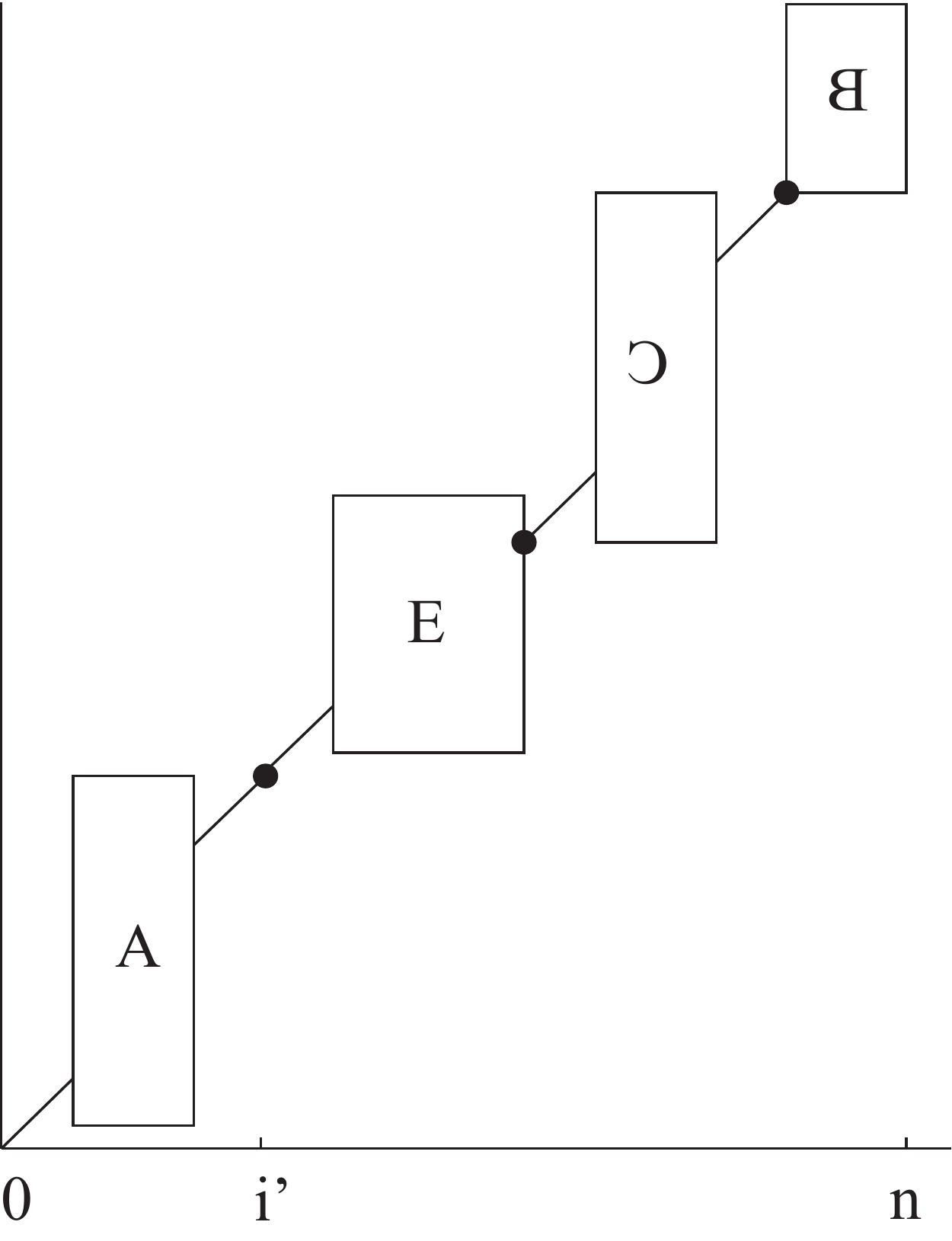}
\includegraphics[width=0.2in, height=1.2in]{arrow.pdf}
\includegraphics[width=1in, height=1.2in]{verification3.pdf}
\caption{First, we do construction of $\Theta$, we get $\bar{D}$ as in $3$rd diagram from $D$ in $1$st diagram. 
Then we do construction of $\Phi$, we get back from $\bar{D}$ to $D$ in $1$st diagram.}
\label{verification}
\end{figure}
We need find $\Phi(\bar{D})$. 
Since in the process of $\Theta$, we shift the path after the $i'$th place to above the line $y=k$,
then the $i'$th place is the largest $t$ such that $a_{t-1}=k-1, a_t=k, a_{t+1}=k+1$.
Thus, $m=i'$.
Now, we cut $\bar{D}$ at $(m,a_m)$, 
then reflect it and glue the resulting path back.
It is easy to see that, 
after the reflection that we did in last step,
$(l,j)$ is the leftmost point that we are looking for in step 3 of the process of $\Phi$, i.e. $(v,a_v)$.
Then after we cut the path in $(l,j)$, we reflect it and glue the resulting path back, which leads to the path $D$.
In the process of $\Theta$, $i'$ is on the left side of $i$ and is the closest to $i$ satisfying $a_{i'-1}=k-1, a_{i'}=k$.
We must choose $m'$ to be on the right side of $m=i'$ and the closest to $m=i'$ satisfying $a_{m'}=k$ and $a_{m'+1}=k-1$.
Since $m=i'$, we have that $m'=i$.
Therefore $\Theta\Phi((D,i))=(D,i)$, as required.
\end{proof}

\subsection{An algebraic proof}\label{subsection2}
Let $C_k(x,q)$ be the generating function for the number of Dyck paths from $(0,0)$ to $(n,0)$ according to the number of steps from height $k$ to height $k-1$, that is,
$$C_k(x,q)=\sum_{n\geq0}\sum_{p\in\mathcal{D}_{(n,0)}}x^nq^{\#st_k(p)},$$
where $st_k(p)$ denotes the number of steps from height $k$ to height $k-1$ in the path $p$. Clearly, the generating function for the number of Dyck paths from $(0,0)$ to $(n,0)$ is given by
\begin{equation}\label{eqlem10}
C_k(x,1)=C(x^2):=\frac{1-\sqrt{1-4x^2}}{2x^2}.
\end{equation}

\begin{prop}\label{lem1}
The generating function $C_k(x,q)$ is given by
$$C_k(x,q)=\frac{U_{k-1}\ttx-qxU_{k-2}\ttx C(x^2)}{x\left[U_k\ttx-qxU_{k-1}\ttx C(x^2)\right]},$$
where $C(x^2)=\frac{1-\sqrt{1-4x^2}}{2x^2}$ and $U_m$ is the $m$-th Chebyshev polynomial of the second kind.
\end{prop}
\begin{proof}
Since each nonempty Dyck path $p$ in $\mathcal{D}_{n,0}$ can be written as $p=up'dp''$, 
where $p'$ is any $1$-shifted Dyck path and $p''$ is any Dyck path ($u$ denotes up-step and $d$ denotes down-step), 
we obtain
\begin{equation}\label{eqlem11}
C_k(x,q)=1+x^2C_{k-1}(x,q)C_k(x,q),\quad k\geq2
\end{equation}
and
\begin{equation}\label{eqlem12}
C_1(x,q)=1+x^2qC(x^2)C_1(x,q),
\end{equation}
where $1$ enumerates the path of length zero.

We now proceed the proof by induction on $k$. Since $U_{-1}(t)=0$, $U_0(t)=1$ and $U_1(t)=2t$, we obtain that \eqref{eqlem12} implies that the proposition holds for $k=1$. Assuming that the claim holds for $k$, we prove it holds for $k+1$. Using \eqref{eqlem11} together with the induction hypothesis we obtain
\begin{align*}
&C_{k+1}(x,q)\\
&=\frac{1}{1-x^2C_k(x,q)}\\
&=\frac{U_k\ttx-qxU_{k-1}\ttx C(x^2)}{U_k\ttx-qxU_{k-1}\ttx C(x^2)-x(U_{k-1}\ttx-qxU_{k-2}\ttx C(x^2))}\\
&=\frac{U_k\ttx-qxU_{k-1}\ttx C(x^2)}{U_k\ttx-xU_{k-1}\ttx-qx(U_{k-1}\ttx-xU_{k-2}\ttx )C(x^2))}.
\end{align*}
Using the fact that Chebyshev polynomials $U_m(t)$ of the second kind satisfy the recurrence relation $U_m(t)=2tU_{m-1}(t)-U_{m-2}(t)$, we get
\begin{align*}
C_{k+1}(x,q)=\frac{U_k\ttx-qxU_{k-1}\ttx C(x^2)}{xU_{k+1}\ttx-qx^2U_{k}\ttx C(x^2)},
\end{align*}
which completes the proof.
\end{proof}

\begin{cor}\cite{F}
\label{cor1}
The number of steps from height $k$ to height $k-1$ in all Dyck paths from $(0,0)$ to $(2n,0)$ is given by
$$\frac{2k+1}{2n+1}\binom{2n+1}{n+k+1}.$$
\end{cor}
\begin{proof}
From \eqref{eqlem10}, \eqref{eqlem11} and \eqref{eqlem12} we get
$$\frac{d}{dq}C_k(x,q)=x^2C_k(x,q)\frac{d}{dq}C_{k-1}(x,q)+x^2C_{k-1}(x,q)\frac{d}{dq}C_k(x,q)$$
and $\frac{d}{dq}C_1(x,q)\mid_{q=1}=x^2C^3(x^2)$. Hence, by induction on $k$ we have that
\begin{equation}\label{eqco1}
\frac{d}{dq}C_k(x,q)\mid_{q=1}=x^{2k}C^{2k+1}(x^2).
\end{equation}
Since the number of steps from height $k$ to height $k-1$ in all generalized  Dyck paths from $(0,0)$ to $(2n,0)$ equals the coefficient of $x^{2n}$ in the generating function $\frac{d}{dq}C_k(x,q)\mid_{q=1}$, we obtain the desired result by \cite[Equation 2.5.16]{Wilf}.
\end{proof}

Let $C_{k,h}(x)$ be the generating function for the number of generalized Dyck paths from $(0,0)$ to $(n,h)$ according to the number steps from height $k$ to height $k-1$, that is,
$$C_{k,h}(x,q)=\sum_{n\geq0}\sum_{p\in\mathcal{D}_{(n,h)}}x^nq^{\#st_k(p)}.$$
In order to an give explicit formula for the generating function $C_{k,h}(x,q)$, we consider the following two cases $k>h$ and $k\leq h$ as follows.

\subsubsection{The case $k>h$}
In this subsection we fix $k$ where $k>h$. 
Using the fact that each nonempty generalized Dyck path $p$ in $\mathcal{D}_{n,h}$ can be decomposed as either $p=up'$ where $p'$ is a $1$-shifted generalized Dyck path from $(0,0)$ to $(n-1,h-1)$, 
or $p'=up'dp''$, where $p'$ is a $1$-shifted Dyck path from $(0,0)$ to $(n',0)$ and $p''$ is a generalized Dyck path from $(n'+1,0)$ to $(n,h)$, 
we obtain
$$C_{k,h}(x,q)=xC_{k-1,h-1}(x,q)+x^2C_{k-1}(x,q)C_{k,h}(x,q),\quad h\geq1$$
and
$$C_{k,0}(x,q)=1+x^2C_{k-1}(x,q)C_{k,0}(x,q).$$
By induction on $h$ we get that the generating function $C_{k,h}(x,q)$ which is given by
\begin{equation}\label{eql23}
C_{k,h}(x,q)=\frac{x^h}{\prod_{j=k-1-h}^{k-1}1-x^2C_j(x,q)}.
\end{equation}

\begin{thm}\label{th2}
Let $k>h\geq0$. The number of steps from height $k$ to height $k-1$ in all generalized Dyck paths from $(0,0)$ to $(n,h)$ is given by
\begin{align*}
\binom{n}{\frac{1}{2}(n-h)+k}-\binom{n}{\frac{1}{2}(n+h)+k+1},
\end{align*}
where $\binom{a}{b}$ is assumed to be $0$ if $a<b$ or if $a,b$ are not nonnegative integers.
\end{thm}
\begin{proof}
From \eqref{eql23}, we obtain that the number of steps from height $k$ to height $k-1$ in all generalized Dyck paths from $(0,0)$ to $(n,h)$ is given by
\begin{align*}
&[x^n]\left(\frac{d}{dq}C_{k,h}(x,q)\right)_{q=1}\\
&=[x^n]\left[\frac{x^h}{\prod_{j=k-1-h}^{k-1}1-x^2C_j(x,1)}\sum_{j=k-1-h}^{k-1}\frac{x^2\frac{d}{dq}C_j(x,q)\mid_{q=1}}{1-x^2C_j(x,1)}\right],
\end{align*}
which, by \eqref{eqlem10} and \eqref{eqco1}, is equivalent to
$$[x^n]\left(\frac{d}{dq}C_{k,h}(x,q)\right)_{q=1}=[x^n]\left(x^{h+2}C^{h+2}(x^2)\sum_{j=k-1-h}^{k-1}x^{2j}C^{2j+1}(x^2)\right).$$
Hence, by \cite[Equation 2.5.16]{Wilf} we get
\begin{align*}
&[x^n]\left(\frac{d}{dq}C_{k,h}(x,q)\right)_{q=1}\\
&=\sum_{j=k-1-h}^{k-1}\frac{2j+h+3}{n+1}\binom{n+1}{\frac{n-2j-h-2}{2}}\\
&=\sum_{j=k-1-h}^{k-1}\binom{n}{\frac{n-2j-h-2}{2}}-\binom{n}{\frac{n-2j-h-4}{2}}\\
&=\binom{n}{\frac{1}{2}(n-h)+k}-\binom{n}{\frac{1}{2}(n+h)+k+1}
\end{align*}
as claimed.
In the second equality, we use
\begin{equation}
\binom{n}{k}-\binom{n}{k-1}=\frac{n-2k+1}{n+1}\binom{n+1}{k}.
\label{binomequation}
\end{equation}
\end{proof}

\subsubsection{The case $k\leq h$}
In this subsection we fix $k$ where $k\leq h$. Using similar arguments as discussed in the above subsection we get that the generating function $C_{k,h}(x,q)$ satisfies
$$C_{k,h}(x,q)=xC_{k-1,h-1}(x,q)+x^2C_{k-1}(x,q)C_{k,h}(x,q),\quad h\geq1$$
and
$$C_{0,h}(x,q)=xC_{0,h-1}(x,q)+x^2C(x^2)C_{0,h}(x,q).$$
By induction on $h$ we get $C_{0,h}(x,q)=x^hC^{h+1}(x^2)$, which, by induction on $k$, implies that
\begin{equation}\label{eql33}
C_{k,h}(x,q)=\frac{x^hC^{h-k+1}(x^2)}{\prod_{j=0}^{k-1}1-x^2C_j(x,q)}.
\end{equation}

\begin{thm}\label{th3}
Let $1\leq k\leq h$. The number of steps from height $k$ to height $k-1$ in all generalized Dyck paths from $(0,0)$ to $(n,h)$ is given by
\begin{align*}
\binom{n}{\frac{n+h}{2}+1}-\binom{n}{\frac{n+h}{2}+k+1},
\end{align*}
where $\binom{a}{b}$ is assumed to be $0$ if $a<b$ or if $a,b$ are not nonnegative integers.
\end{thm}
\begin{proof}
From \eqref{eql33}, we obtain that the number of steps from height $k$ to height $k-1$ in all generalized Dyck paths from $(0,0)$ to $(n,h)$ is given by
\begin{align*}
&[x^n]\left(\frac{d}{dq}C_{k,h}(x,q)\right)_{q=1}\\
&=[x^n]\left[\frac{x^hC^{h-k+1}(x^2)}{\prod_{j=0}^{k-1}1-x^2C_j(x,1)}\sum_{j=0}^{k-1}\frac{x^2\frac{d}{dq}C_j(x,q)\mid_{q=1}}{1-x^2C_j(x,1)}\right],
\end{align*}
which, by \eqref{eqlem10} and \eqref{eqco1}, is equivalent to
$$[x^n]\left(\frac{d}{dq}C_{k,h}(x,q)\right)_{q=1}=[x^n]\left(x^{h+2}C^{h+3}(x^2)\sum_{j=0}^{k-1}x^{2j}C^{2j}(x^2)\right).$$
Hence, by \cite[Equation 2.5.16]{Wilf} we obtain
\begin{align*}
&[x^n]\left(\frac{d}{dq}C_{k,h}(x,q)\right)_{q=1}\\
&=\sum_{j=0}^{k-1}\frac{h+2j+3}{n+1}\binom{n+1}{\frac{n-2j-h-2}{2}}\\
&=\sum_{j=0}^{k-1}\binom{n}{\frac{n-2j-h-2}{2}}-\binom{n}{\frac{n-2j-h-4}{2}}\\
&=\binom{n}{\frac{n+h}{2}+1}-\binom{n}{\frac{n+h}{2}+k+1}
\end{align*}
as claimed.
In the second equality, we use the Equation \ref{binomequation}.
\end{proof}

\section{A new semi-meander determinant}\label{ANSD}
In \cite{F}, Di Francesco defined a {\em semi-meander determinant}.
Here, we will present a different bilinear form on the same module.
We will calculate the Gram determinant of this new form with respect to a natural basis.

\begin{de}\cite{F}
A semi-meander of order $n$ with winding number $h$ is a planar configuration of non-selfintersecting loops
crossing the positive half line through $n$ distinct points and negative half line through $h$ distinct points such that no two points from the set of $h$ points are contiguous on a loop.
We consider such diagrams up to smooth deformations preserving the topology of the configuration.
\end{de}

We can cut the semi-meander into an upper and a lower diagram as described in Figure \ref{semimeander}.
\begin{figure}[h]
\includegraphics[width=1.2in, height=1in]{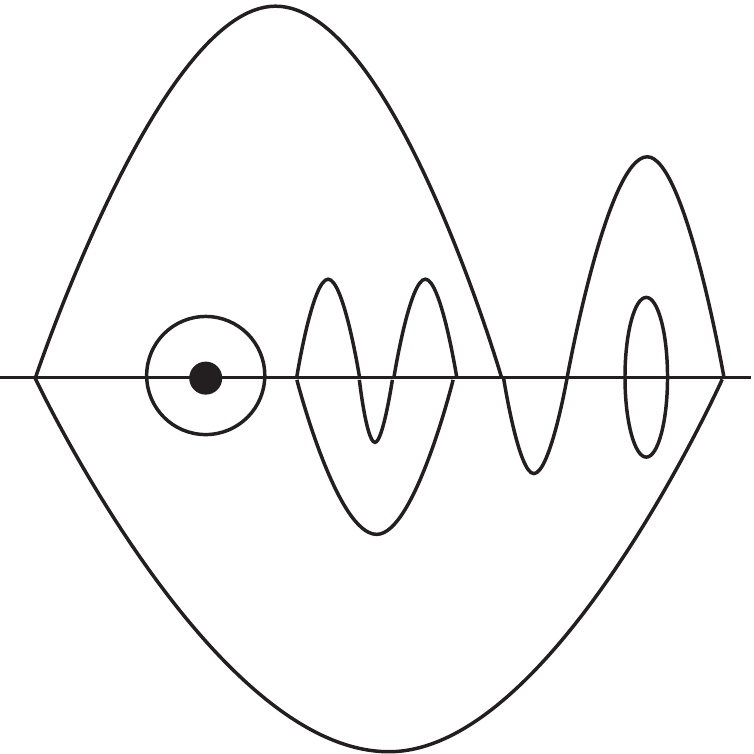}
\includegraphics[width=0.2in, height=1in]{varrow.pdf}
\includegraphics[width=1.2in, height=1in]{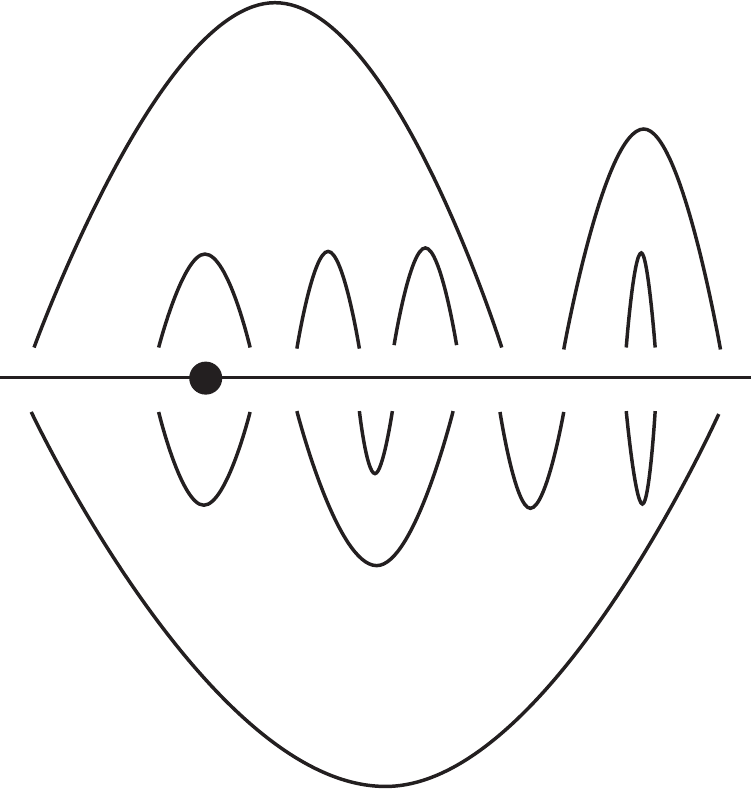}
\includegraphics[width=0.2in, height=1in]{varrow.pdf}
\includegraphics[width=1.2in, height=1in]{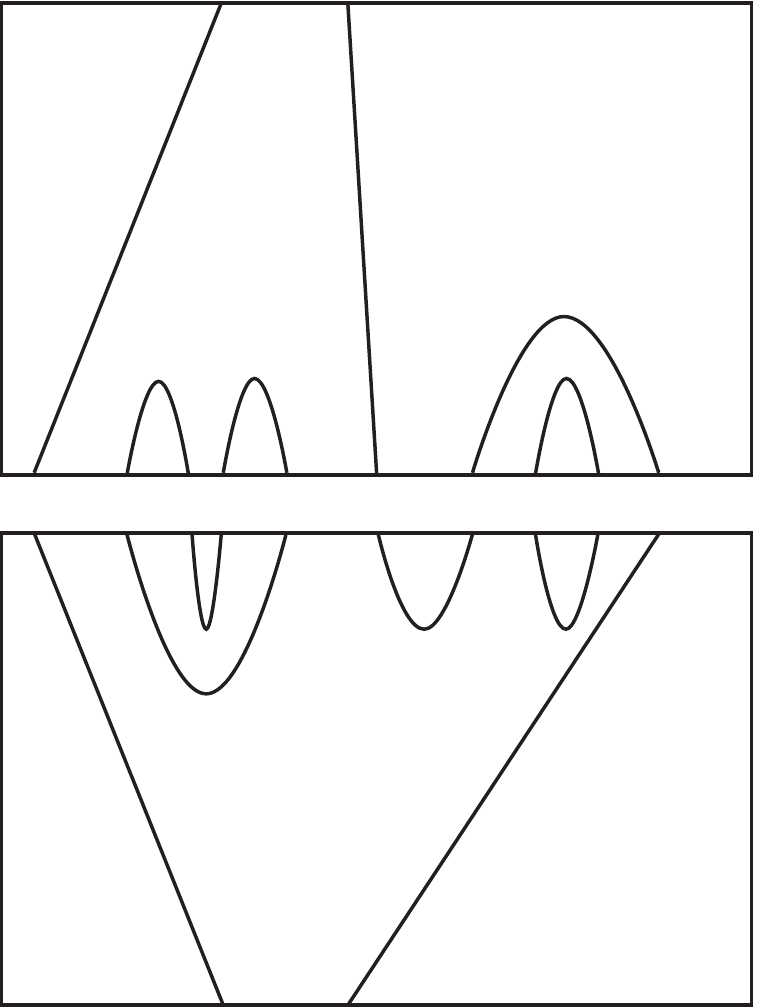}
\caption{A semi-meander of order $n=10$ with winding number $h=2$}
\label{semimeander}
\end{figure}

\begin{de}
\label{epsilonnh}
Let $B'_{a_1,\ldots,a_n}$ be the diagram $B_{a_1,\ldots,a_n}$ with the idempotent on $a_n$ removed.
We denote $\ce_n^h$ to be $span_{\Lambda}\{B'_{a_1,\ldots,a_n}\}_{(a_1,\ldots,a_n)}\subset\cs(I\times I,n+h)$.
\end{de}

In \cite{F}, Di Francesco defined the following matrix:
\begin{de}\cite{F}
$T=[T_{a,b}]$ with $T_{a,b}=\delta^{c(a,b)}$,
where $a,b\in\{B'_{a_1,\ldots,a_n}\}_{(a_1,\ldots,a_n)}$ and $c(a,b)$ is the number of the components of the semimeander by gluing $a$ to $b$.
\end{de}

\begin{rem}
We can define a bilinear form on $\ce_n^h$ by extending the map $f(a,b)=\delta^{c(a,b)}$ with $a,b\in\{B'_{a_1,\ldots,a_n}\}_{(a_1,\ldots,a_n)}$ bilinearly.
The matrix $T$ defined by Di Francesco is the matrix of this bilinear form with respect to the basis $a,b\in\{B'_{a_1,\ldots,a_n}\}_{(a_1,\ldots,a_n)}$.
\end{rem}

Now, we define a different matrix as follows:
\begin{de}
Let $S=[S_{a,b}]$ with $S_{a,b}=\delta^{c(a,b)}$,
where $a,b\in\{B'_{a_1,\ldots,a_n}\}_{(a_1,\ldots,a_n)}$,
and $c(a,b)$ is the number of the components of the semi-meander obtained by gluing $a$ to $b$
if the $h$ intersection points on the negative half line belong to $h$ distinct components of the resulting collection of loops;
otherwise, $S_{a,b}$ is $0$.
\end{de}

\begin{thm} 
We have
\begin{equation}
\det(S)=(\frac{\Delta_1^h}{\Delta_h})^{|\cd_n^h|}\det(B)
=\Delta_1^{|\cd_n^h|}\prod_{k}(\frac{\Delta_k}{\Delta_{k-1}})^{\alpha^k_{(n,h)}}.
\notag
\end{equation}
\end{thm}
\begin{proof}
Let $\ce_n^h$ be the subspace of $\cs(D^2,n+h)$ defined in Definition \ref{epsilonnh}.
Just as in $TL_n$, we define a map $L$ on $\{B'_{a_1,\ldots,a_n}\}_{(a_1,\ldots,a_n)}\times\{B'_{a_1,\ldots,a_n}\}_{(a_1,\ldots,a_n)}$ by connecting two elements in $\{B'_{a_1,\ldots,a_n}\}_{(a_1,\ldots,a_n)}$ with $n+h$ parallel strings.
If the $h$ points on $I\times\{1\}$ belong to $h$ different components, then we evaluate the resulting diagram by Kauffman bracket. 
Otherwise, we make it $0$.
Then we extend this map to a bilinear form on $\ce_n^h$.
It is easy to see that the matrix of $L$ with respect to the basis $\{B'_{a_1,\ldots,a_n}\}_{(a_1,\ldots,a_n)}$ is equal to $S$.
Moreover,
\begin{equation}
G(B_{(a_1,\ldots,a_n)},B_{(b_1,\ldots,b_n)})=\frac{\Delta_1^h}{\Delta_h}L(B'_{(a_1,\ldots,a_n)},B'_{(b_1,\ldots,b_n)}),
\notag
\end{equation}
for all $(a_1,\ldots,a_n),(b_1,\ldots,b_n)$. 
Then the result follows easily from Theorem \ref{main}.
\end{proof}

\section*{Acknowledgements}
The first author thanks his advisor Professor Gilmer for helpful discussions.
The first author was partially supported by research assistantship funded by NSF-DMS-0905736.


\end{document}